\documentclass[10pt, reqno]{amsart}
\pdfoutput=1 
\usepackage{amsmath, amssymb, amsthm, enumerate, bbm,soul}
\usepackage[colorlinks, linkcolor=blue, citecolor=magenta, hypertexnames=false,backref]{hyperref} 
\usepackage[dvipsnames]{xcolor}

\usepackage{graphicx}
\usepackage{booktabs}
\usepackage[font=small]{caption} 
\usepackage[flushleft]{threeparttable}
\usepackage{mathtools, breqn}
\usepackage{url}

\def\spine{1.1in}
\usepackage[
heightrounded,
top=1in,
bottom=1.3in,
inner=\spine,
outer=\spine 
]{geometry} %

\usepackage[final, expansion=true, protrusion=true, spacing=true, kerning=true]{microtype}
\microtypecontext{spacing= nonfrench}

\def\dim{\operatorname{dim}}

\def\sp{\operatorname{span}}

\def\theta{\vartheta}
\def\Tr{\operatorname{Tr}}
\def\sgn{\operatorname{sgn}}
\def\adj{\operatorname{adj}}

\numberwithin{equation}{section}

\newtheorem{theorem}{Theorem}[section]
\newtheorem{lemma}[theorem]{Lemma}

\newtheorem{proposition}[theorem]{Proposition}
\newtheorem{corollary}[theorem]{Corollary}

\theoremstyle{definition}

\newtheorem{remark}{Remark}

\title{Optimal Measures for Multivariate Geometric Potentials}
\date{\today}
\author[D. Bilyk]{Dmitriy Bilyk}
\author[D. Ferizovi\'{c}]{Damir Ferizovi\'{c}}
\author[A. Glazyrin]{Alexey Glazyrin}
\author[R. Matzke]{Ryan W. Matzke}
\author[J. Park]{Josiah Park}
\author[O. Vlasiuk]{Oleksandr Vlasiuk}
\address{School of Mathematics, University of Minnesota, Minneapolis, MN 55455} 
\email{dbilyk@math.umn.edu}
\address{Department of Mathematics, Katholieke Universiteit Leuven, Leuven, Belgium} 
\email{damir.ferizovic@kuleuven.be}
\address{School of Mathematical \& Statistical Sciences, The University of Texas Rio Grande Valley, Brownsville, TX 78520}
\email{alexey.glazyrin@utrgv.edu}
\address{Department of Mathematics, Vanderbilt University, Nashville, TN 37240} 
\email{ryan.w.matzke@vanderbilt.edu}
\address{School of Mathematics, Georgia Insititute of Technology, Atlanta, GA 30332}
\email{j.park@gatech.edu}
\address{Department of Mathematics, Vanderbilt University, Nashville, TN, 37240}
\email{oleksandr.vlasiuk@vanderbilt.edu }

\keywords{Potential energy minimization, optimal measures, random polytopes, spherical codes, tight frames, isotropic measures}

\begin{document} 
\maketitle
\begin{abstract}
We study measures and point configurations  optimizing energies based on multivariate potentials. The emphasis is put on potentials defined by geometric characteristics of sets of  points, which serve as multi-input generalizations of the   well-known Riesz potentials for pairwise interaction. One of such potentials is volume squared of the simplex with vertices at the $k \ge 3$ given points: we show that the arising energy is maximized by balanced isotropic measures, in contrast to the classical two-input energy. These results are used to obtain interesting geometric optimality properties of the regular simplex. As the main machinery, we adapt the semidefinite programming method to this context and establish relevant versions of the $k$-point  bounds.
\end{abstract}
\tableofcontents

\section{Introduction}

A variety of problems in many areas of mathematics and science can be formulated as discrete or continuous energy optimization problems for two-point interaction  potentials. The discrete energy  and the continuous energy integral in this setup are defined as
\begin{equation}\label{e.2ener}
\frac{1}{N^2} \sum_{x,y \in \omega_N}  K (x,y) \,\,\, \textup{ or }   \int_\Omega \int_\Omega K(x,y) \, d\mu (x) \, d\mu (y),
\end{equation}
where $K: \Omega \times \Omega \rightarrow \mathbb R$ is a potential function. In the former case, the energy is determined for discrete sets $\omega_N$ of $N$ points in $\Omega$. In the latter case, it  is determined for probability measures $\mu$ on the domain $\Omega$.  For $\Omega \subset \mathbb R^d$, undoubtedly, one of the most well studied energies of this type are the Riesz energies with the kernels $K(x,y) = \| x-y \|^s$ (diagonal terms need to be dropped  in the discrete case for $s<0$). We refer the reader to \cite{BHS} for an excellent exposition of the subject. 

However, numerous applications (e.g. Menger curvature \cite{MMV}, $U$-statistics \cite{L,V}, $k$-point bounds \cite{BV, DMOV, Mu, CW}, three-nucleon force in physics \cite{Z}, etc) call for energies that depend on interactions of triples or $k$-tuples of particles, rather than just pairwise interactions, i.e.\ energies of the type
\begin{align}\label{e.nener}
E_K (\omega_N) &  = \frac{1}{N^k} \sum_{ z_1,\dots,z_k   \in \omega_N}  K (z_1,\ldots,z_k), \\  
\label{e.neneri} I_K (\mu) & =   \int_\Omega \dots \int_\Omega K (x_1,\dots,x_k) \, d\mu (x_1) \,\dots \,  d\mu (x_k),
\end{align}
with $k\ge 3$. The question of interest is  finding  point configurations and measures optimizing such energies.\\

Continuing the general study initiated in \cite{BFGMPV1}, in this paper  we study multivariate potentials that are determined by geometric characteristics of  sets of $k$ points in $\mathbb R^d$ and, at the same time, serve as generalizations of classical pairwise potentials ubiquitous in the literature, in particular, the aforementioned Riesz potentials.  There are two main classes of such potentials which we investigate here. \\

\noindent {\bf The potential $V$: volume of the parallelepiped spanned by $k$-tuples of vectors.} Let $2\leq k\leq d$ and define the $k$-input kernel $V(x_1,\ldots, x_k)$ as the $k$-dimensional volume of the parallelepiped whose vertices are the points $x_1$, $\ldots$, $x_k$ and the origin. In other words, the parallelepiped spanned by the {\emph{vectors}}   $x_1$, $\ldots$, $x_k$ (or, equivalently, $V$ is the volume of the simplex spanned by these $k$ vectors, scaled by a factor of $k!$). Note that $V^2$ is the determinant of the Gram matrix of the set of vectors $\{x_1,\ldots, x_k\}$. \\

\noindent {\bf The potential $A$: volume (area) of the simplex spanned by $k$-tuples of points.} For $2\leq k\leq d+1$, define the $k$-input kernel $A(x_1,\ldots, x_k)$ as the $(k-1)$-dimensional volume of the simplex whose vertices are the points $x_1$, $\ldots$, $x_k$.  Similarly to $V^2$, the potential $A^2$ can be represented as a determinant of a matrix based on scalar products of the set of vectors $\{x_1,\ldots, x_k\}$, see Lemma \ref{lem:A-formula}.\\ 

Observe that in the case $k=3$, $V$ is simply the three-dimensional {\emph{volume}}  of the parallelepiped  spanned by the vectors $x_1$,  $x_2$, $x_3$, while $A$ is the {\emph{area}} of the triangle with vertices at the points $x_1$,  $x_2$, $x_3$. This explains the notation chosen for these potentials. The three-input case of the potentials will be the focus of Section \ref{sec:SD}.  Note also that for $k=2$, $A(x,y) = \| x -y\|$, so the potentials $A^s$ are direct multivariate generalizations of the Riesz potentials. This was also remarked upon in \cite{KPS}, where the authors studied the gradient flow of $A^2$ as a generalization of the linear consensus model. The  two-input case for both of these potentials is discussed in Section \ref{sec:k=2}. 
\\

In direct analogy to the Riesz energies, we shall study multi-input energies with kernels given by powers of these potentials: $A^{s}$ and $V^s$ (in this paper, the powers are mostly assumed to be positive, i.e. $s>0$). Due to the nature of these potentials for $s>0$, one is generally interested in measures and point configurations that {\emph{maximize}} (rather than minimize) the corresponding energies. The geometric setting will be primarily restricted to the case when the domain $\Omega$ is the unit sphere $\mathbb S^{d-1}$, as well as $\Omega= \mathbb R^d$ with certain moment restrictions on the underlying probability measures \eqref{eq:moment}. In the former case, one of the most natural questions is whether the normalized uniform surface measure  $\sigma$ on the sphere $\mathbb S^{d-1}$ maximize the energies $I_{A^s} (\mu)$ and $I_{V^s} (\mu)$.

These questions can be reformulated in probabilistic terms as follows: assume that $k$ random points/vectors are chosen on the sphere $\mathbb S^{d-1}$ independently according to the probability distribution $\mu$. Which probability distribution $\mu$ maximizes the expected $s^{th}$ power of the volume of the parallelepiped spanned by the vectors (or, respectively, of the volume of the simplex spanned by the random points)? Is the uniform distribution $\sigma$ optimal?  \\

The case $s=2$ appears to be more manageable than others, since, as mentioned above, both $V^2$ and $A^2$ can be expressed as polynomials. In fact, it has been already shown by Cahill and Casazza \cite{CC} (see also Theorem \ref{thm:volume-gen} below) that $I_{V^2}$ is maximized by isotropic measures on the sphere (see \eqref{eq:isotropic_init} for the definition), which includes $\sigma$. Based on this result  we show that $I_{V^s}$ is maximized by the discrete measure uniformly distributed over the vertices of an orthonormal basis when $s>2$ (Corollary \ref{cor:V^s for s>2}).  Other main results of the present paper concerning multivariate geometric potentials  include:

\begin{itemize}
\item  The  maximizers of the energy $I_{A^2}(\mu)$ on the sphere $\mathbb S^{d-1}$    are exactly the balanced isotropic measures (which includes the uniform surface measure $\sigma$, see Section \ref{sec:not} for the relevant definitions). This is proved in Theorem \ref{thm:area-gen} in full generality (for all $d\ge 2$ and $3\le k \le d+1$), but different proofs of partial cases are also given in Theorem \ref{thm: triangle area squared max} (the case of $k=3$ inputs, i.e. area squared of a triangle) and Theorem \ref{thm:A^2maxGen} ($k=d+1$ inputs in dimension $d\ge 2$, i.e. a full-dimensional simplex; this theorem also applies to measures on $\mathbb R^d$ with unit second moment). 
\item When $k=d+1$ and $s>2$, the energy $I_{A^s}(\mu)$ is maximized by the discrete measure uniformly distributed over the vertices of a regular  $d$-dimensional simplex 
(Corollary \ref{cor:A^s for s>2}). 
\item For $0<s\le 2$, the discrete energies $E_{V^s}$ and $E_{A^s}$ with $N=d+1$ points are maximized by the vertices of a regular $d$-dimensional simplex, see Corollary \ref{cor:Simplex Best A, V}.  As a corollary, a regular $d$-dimensional simplex maximizes the sum of volumes of $j$-dimensional faces ($1\le j \le d$) among all simplices of a given circumradius (Corollary \ref{cor:Geometric Simplex is best}). 
\end{itemize}
\noindent For more precise technical statements of these results the reader is directed to the theorems and corollaries  referenced above. \\

The case $s=2$ is also special due to the fact that in the classical two-input setting  this is exactly the phase transition for the Riesz energy on the sphere, which is maximized uniquely by the uniform surface measure $\sigma$ for $0<s<2$ and by discrete measures for $s>2$ \cite{Bj} (see also Proposition \ref{prop:k=2} below).  Some of our main results suggest that similar behavior persists in the multivariate case, although the case $0<s<2$ (including the very natural $s=1$) remains out of reach. We conjecture that the uniform surface  measure $\sigma$ maximizes both $I_{A^s} (\mu)$ and $I_{V^s} (\mu)$  when $0<s<2$. \\

The main machinery for our optimization results is a variant of the semidefinite programming method. We adapt the method developed by Bachoc and Vallentin for finding three-point packing bounds for spherical codes \cite{BV}. Three-point bounds were  also applied to energy optimization for pair potentials in \cite{CW} and for multivariate $p$-frame energy in \cite{BFGMPV2}.  The approach of Bachoc and Vallentin later was generalized by Musin \cite{Mu} who established the $k$-point version of packing bounds. This method is actively utilized for solving packing/energy problems (see, e.g. \cite{DMOV, DDM}) but its applicability is typically limited due to complexity of actual semidefinite programs. Our paper seems to be the first one where general $k$-point bounds are explicitly used for all positive integer $k$.\\

The paper is organized as follows.  Section \ref{sec:not}   describes the relevant background, definitions, notation, and covers the two-input case of the energies.   Section \ref{sec:Discrete A and V}  presents the applications of our main results to some geometric  optimality properties of the regular simplex. 
In Section \ref{sec:SD} we discuss  the semidefinite programming approach  of \cite{BV} and demonstrate  how it leads to optimization results for $3$-input energies with geometric kernels. 
Section \ref{sec:vol} shows how the known results about $I_{V^2}$ \cite{CC} can be used to obtain partial results for $I_{A^2}$, as well as the discreteness of maximizers  for $I_{V^s}$ and $I_{A^s}$ with $s>2$. 
In Section \ref{sec:k-point} we provide a self-contained description of $k$-point semidefinite bounds for the sphere and give a  general construction of $k$-positive definite multivariate functions based on these bounds. Finally, in  the main result of Section \ref{sec:A^2 on Sphere}, we use multivariate functions from Section \ref{sec:k-point} to prove that the energy $I_{A^2}$ based on the squared volume of a simplex is maximized by balanced isotropic measures on the sphere. In the Appendix (Section \ref{sec:Appen A^2}) we give an explicit expression for the potential $A^2$.   

\section{Background and notation}\label{sec:not}

The notation in this paper generally follows \cite{BFGMPV1}. Most of the optimization problems, with a few exceptions, will be formulated for measures or finite configurations of points on the unit Euclidean sphere $\mathbb{S}^{d-1}$. Often the potentials will be invariant under the action changing an argument to its opposite. Essentially, this means that the underlying space is the real projective space $\mathbb{RP}^{d-1}$, but we will still formulate our results in terms of the unit sphere.

In what follows,  the domain $\Omega$ is either the sphere $\mathbb S^{d-1}$ or the Euclidean space $\mathbb R^d$. 
Assume  $ k \in \mathbb{N}\setminus\{1\}$ is the number of inputs and the kernel $K: \Omega^k \rightarrow \mathbb{R}$ is continuous. We denote by $\mathcal{M}(\Omega)$ the set of finite signed Borel measures on $\Omega$, and by $\mathcal{P}(\Omega)$ the set of Borel probability measures on $\Omega$. If $\Omega = \mathbb{R}^d$, we define $\mathcal{P}^*(\mathbb R^d)$ to be the set of  Borel probability measures $\mu$  on $\mathbb R^d$ satisfying
\begin{equation}\label{eq:moment}
\int_{\mathbb R^d}  \| x \|^2  d \mu(x) = 1.
\end{equation}
Observe that, by a slight abuse of notation, $\mathcal{P}(\mathbb S^{d-1}) \subset \mathcal{P}^*(\mathbb R^d)$ 

Let $\omega_N = \{ z_1, z_2, \ldots, z_N\}$ be an $N$-point configuration (multiset) in $\Omega$, for $N \geq k$. Then the discrete $K$-energy of $\omega_N$ is defined to be 
\begin{equation*}\label{eq:DiscreteEnergyDef}
E_K(\omega_N) :=  \frac{1}{N^k} \sum_{j_1=1}^{N} \cdots \sum_{j_k = 1}^{N} K(z_{j_1}, \ldots, z_{j_k}).
\end{equation*}
Similarly, we define the  energy integral  for measures on $\Omega$: for $\mu \in \mathcal{M}(\Omega)$, \begin{equation*}\label{eq:ContEnergyDef}
I_K(\mu ) = \int_{ \Omega} \cdots \int_{\Omega} K(x_1, \ldots, x_k) \,d\mu(x_1) \cdots  d\mu(x_k),
\end{equation*} 
when absolutely convergent, as will be the case in all of the contexts considered below.
In the present paper we shall be interested in finding probability measures ($\mu \in \mathcal{P}(\mathbb S^{d-1} ) $ or $\mu \in \mathcal{P}^*(\mathbb R^d)$) which optimize (in most cases, maximize) the energy integrals $I_K$.

\subsection{Isotropic measures and frame energy} 
The \textit{$p$-frame potential} is defined as $|\langle x,y \rangle|^p$. The notion of the $2$-frame potential, or simply \textit{frame potential}, was introduced by Benedetto and Fickus \cite{BF}, and later generalized to $p \in (0, \infty)$ by Ehler and Okoudjou \cite{EO}.  Minimization of the frame energy is well understood: the following lemma is usually stated for $\mu \in \mathcal P (\mathbb S^{d-1})$, see e.g. Theorem 4.10 in \cite{BF}, but the extension to $ \mathcal{P}^*(\mathbb R^d)$ is straightforward (see also Remark \ref{rem:proj} below).

\begin{lemma}\label{lem:frame}
For 
any {$\mu\in \mathcal{P}^*(\mathbb R^d)$}, and hence also any $\mu \in \mathcal P (\mathbb S^{d-1})$,
\begin{equation*}
\int_{ \Omega} \int_{\Omega} \langle x,y \rangle^2\, d\mu(x) d\mu(y) \geq \frac 1 d.
\end{equation*}
\end{lemma}

It is easy to see that the equality in the estimate above is achieved precisely for the measures which  satisfy
\begin{equation}\label{eq:isotropic_init}
\int_{ \Omega} x x^T \,  d\mu(x)=\frac 1 d I_d,
\end{equation}
where $I_d$ is the  $d\times d$ identity  matrix. It will be convenient for us to use this condition in the following form: for any $y\in\mathbb{S}^{d-1}$,
\begin{equation}\label{eq:isotropic}
\int_{ \Omega} \langle x, y\rangle^2\,  d\mu(x)=\frac 1 d.
\end{equation}
Measures which satisfy \eqref{eq:isotropic_init} or, equivalently, \eqref{eq:isotropic}, are called \textit{isotropic}.  We note that $\operatorname{Tr} (x x^T)=\|x\|^2$ and (\ref{eq:isotropic_init}) implies $\int_{\Omega}  \| x \|^2  d \mu(x) = 1$ so, as a matter of fact, all isotropic measures on $\mathbb{R}^d$ automatically belong to $\mathcal{P}^*(\mathbb{R}^d)$.

The discrete version of Lemma \ref{lem:frame} states  that for $N \geq d$ and $\{ x_1,\dots,x_N\} \subset \mathbb S^{d-1}$,
\begin{equation*}
\sum_{i=1}^{N}\sum_{j = 1}^{N} \langle x_i, x_j \rangle^2\geq \frac {N^2} d.
\end{equation*}
Discrete sets for which this bound is sharp are known as \textit{unit norm tight frames}, which explains the term  \emph{frame energy}. The lower bound  for the discrete frame energy is a special case of bounds by Welch \cite{W} and Sidelnikov \cite{Sid}.

There is a natural projection $\pi:\mathcal{P}^*(\mathbb{R}^d)\rightarrow \mathcal{P}(\mathbb{S}^{d-1})$ that maps isotropic measures in $\mathbb{R}^d$ onto isotropic measures in $\mathbb{S}^{d-1}$. First, we define the projection $\pi_0:\mathbb{R}^d\setminus\{0\}\rightarrow\mathbb{S}^{d-1}$ by $\pi_0(x)=x/\|x\|$. Now for any $\mu\in \mathcal{P}^*(\mathbb{R}^d)$, we define $\mu^*=\pi(\mu)$ as {the pushforward measure $ (\pi_0)_\# \|x\|^2\, d\mu(x) $, that is}: for any Borel subset $B$ of $\mathbb{S}^{d-1}$, we set
$$\mu^*(B)=\int_{\pi_0^{-1}(B)} \|x\|^2\, d\mu(x).$$
Clearly, $\mu^*$ is a Borel probability measure on $\mathbb{S}^{d-1}$. Checking (\ref{eq:isotropic}), we can also see that for an isotropic $\mu$, $\pi(\mu)$ is isotropic too.

\begin{remark}\label{rem:proj}
For potentials $K$ that are homogeneous of degree $2$ in each variable, the energy $I_K(\mu)$ is invariant under the projection $\pi$. The kernel $V^2$ is such a function, since it is the determinant of the Gram matrix of $\{ x_1,\dots,x_N\}$. This property is also satisfied by the frame potential $K(x,y) = \langle x,y \rangle^2$.  In such cases, it is sufficient to find optimizers for probability  measures on the sphere in order to solve an optimization problem in $\mathcal{P}^*(\mathbb{R}^d)$.
\end{remark}

We call a measure $\mu$ \textit{balanced} if $\int_{\Omega} x\, d\mu(x) = 0$, i.e. the center of mass is at the origin. Balanced isotropic measures can be used to construct isotropic measures in higher dimensions, as will be seen in the proof of  Theorem~\ref{thm:A^2maxGen}.

\subsection{Linear programming and positive definite kernels}
The linear programming method, developed for the spherical case in \cite{DGS},  appeared to be successful in finding optimizing measures and point configurations as well as in giving lower bounds for two-point interaction energies  (see, e.g., \cite{BGMPV1,CK, Y}). Here we briefly describe how it works. In Sections \ref{sec:SD} and \ref{sec:k-point}, we explain in more detail how the method is extended to semidefinite bounds for $k$-point energies.

A symmetric kernel $K: ({\mathbb S^{d-1}})^2 \rightarrow \mathbb{R}$ is called \textit{positive definite} if for every $\nu \in \mathcal{M}(\mathbb S^{d-1})$, the energy integral satisfies  $I_K(\nu) \geq 0$. A classical theorem of Schoenberg described positive definite kernels via Gegenbauer polynomials \cite{Sch}. The Gegenbauer polynomials $P_m^d$, $ m\geq 0 $, form an orthogonal basis on $[-1,1]$ with respect to the measure $(1-t^2)^{\frac{d-3}2}dt$. Here, $P_m^d$ is normalized so that $P_m^d(1) = 1$. All continuous functions on $[-1,1]$ can be expanded like so:
\begin{equation}\label{eq:GegenbauerExpansion}
f(t)=\sum_{m=0}^{\infty} \hat{f}_m P_m^d(t),
\end{equation}
where the sum converges uniformly and absolutely if $K(x,y) = f( \langle x, y \rangle)$ is positive definite on $\mathbb{S}^{d-1}$ (due to Mercer's Theorem). Rotationally-invariant positive definite kernels on the sphere are exactly {characterized} by the positivity of their Gegenbauer coefficients.

\begin{theorem}[Schoenberg \cite{Sch}]\label{thm:sch}
The kernel $K(x,y)=f(\langle x,y \rangle)$ is positive definite on $\mathbb{S}^{d-1}$ if and only if all coefficients $\hat{f}_m$ of the Gegenbauer expansion \eqref{eq:GegenbauerExpansion}  are non-negative.
\end{theorem}

More background on Gegenbauer polynomials, energy, and positive definite kernels on the sphere can be found in \cite{AH, BHS}.

If one  can bound a given function $f$ from below by a positive definite (modulo a constant) function $h$, usually a polynomial,  then the  linear programming bounds on the energy of $f$ are then essentially consequences of the inequalities
$$\int_{\mathbb{S}^{d-1}}\int_{\mathbb{S}^{d-1}} P_m^d(\langle x,y \rangle) \,d\mu(x) d\mu(y)\geq 0.$$
For example, $P_2^d(t)=\frac {dt^2-1} {d-1}$ and the inequality above immediately implies the lower bound in Lemma \ref{lem:frame}. \\

\subsection{$k$-positive definite kernels} As an extension of the notion of positive definite kernels to the multivariate case, we define \textit{$k$-positive definite kernels}. Let $K: (\mathbb S^{d-1})^k \rightarrow \mathbb{R}$ be continuous and symmetric in the first two variables. We define the \textit{potential function of $K$ for fixed $z_3, \ldots, z_k$} as
\begin{equation}\label{eq:PotentialFuncDef}
U_{K}^{z_3, \ldots, z_{k}}(x,y):= K(x,y, z_3, \ldots, z_k).
\end{equation} 
We call $K$ $k$-positive definite if for any $z_3, \ldots, z_k \in \mathbb S^{d-1}$,  the   potential function $U_{K}^{z_3, \ldots, z_{k}}(x,y)$ is positive definite as a function of $x$ and $y$. For kernels symmetric in all variables, this definition is the same as the one given in \cite{BFGMPV1}. A kernel $Y$ is $k$-negative definite if $-Y$ is $k$-positive definite. In Section \ref{sec:k-point} we provide a self-contained construction of large classes of $k$-positive definite kernels for $\mathbb{S}^{d-1}$. Here we collect some general results about positive definiteness and energy minimization for multivariate kernels.

\begin{lemma}\label{lem:Schur's Lemma}
Suppose that $K_1, K_2, \ldots$ are $k$-positive definite. Then $K_1 + K_2$ and $K_1 K_2$ are $k$-positive definite. If the sequence of $K_j$'s converges (uniformly in the first two variables and pointwise in the others) to a kernel $K$, then $K$ is also $k$-positive definite.
\end{lemma}

This result follows immediately from the same results for two-input kernels. Similarly, we have the following:

\begin{lemma}\label{lem:Schur's Lemma2}
Suppose that $K_1, K_2, \ldots$ are kernels such that each $I_{K_j}$ is minimized by some probability measure $\mu$. Then $I_{K_1 + K_2}$ is also. If the sequence of $K_j$'s converges (uniformly in the first two variables and pointwise in the others) to a kernel $K$, then $I_K$ is also minimized by $\mu$.
\end{lemma}
As in the two-input case, multiplication does not generally preserve the minimizers of energies.

\begin{proposition}\label{prop:kPosDef and Energy Zero is Min}
Suppose that $Y$ is a $k$-positive definite kernel on $\mathbb S^{d-1}$ and $\mu \in \mathcal{P}(\mathbb S^{d-1})$ with $I_Y(\mu) = 0$. Then $\mu$ is a minimizer of $I_Y$.
\end{proposition}

\begin{proof}
Let $\nu \in \mathcal{P}(\mathbb S^{d-1})$. Then, since $Y$ is $k$-positive definite
\begin{equation*}
I_Y(\nu) \geq \min_{z_3, \ldots, z_k \in \mathbb S^{d-1}} \int_{\mathbb S^{d-1}} \int_{\mathbb S^{d-1}} Y(x,y, z_3, \ldots, z_n)\, d\nu(x) d\nu(y) \geq 0 = I_Y(\mu).
\end{equation*}
\end{proof}

We can create multivariate kernels from kernels with fewer inputs in a natural way that preserves minimizers of the energy.
\begin{lemma}\label{lem:kPosDef to nPosDef}
For some kernel $Y: (\mathbb S^{d-1})^k \rightarrow \mathbb{R}$ and $n > k$, let
\begin{equation}
K(x_1, x_2, \ldots, x_n) = \frac{1}{|S|} \sum_{\pi \in S} Y(x_1, x_2, x_{\pi(3)}, \ldots, x_{\pi(k)}),
\end{equation}
where $S$ is a nonempty set of permutations of  the set $\{ 3, \ldots, n\}$. Then $I_K$ is minimized by $\mu \in \mathcal{P}(\mathbb S^{d-1})$ if and only if $I_Y$ is as well. In addition, if $Y$ is $k$-positive definite, then $K$ is $n$-positive definite.
\end{lemma}

Note that if $S$ is the set of all such permutations, then $K$ is symmetric in the last $n-2$ variables.

\begin{proof}
For any $\nu \in \mathcal{M}(\mathbb S^{d-1})$, we see that
\begin{equation*}
    I_K(\nu) = I_Y(\nu),
\end{equation*}
meaning their minimizers must be the same,  and for any $z_3, \ldots, z_n \in \mathbb S^{d-1}$,
\begin{equation*}
\int_{\mathbb S^{d-1}} \int_{\mathbb S^{d-1}} K(x,y, z_3, \ldots, z_n)\, d\nu(x) d\nu(y) = \frac{1}{|S|} \sum_{\pi \in S} \int_{\mathbb S^{d-1}} \int_{\mathbb S^{d-1}} Y(x, y, z_{\pi(3)}, \ldots, z_{\pi(k)})\nu(x)\, d\nu(y),
\end{equation*}
which is non-negative if $Y$ is $k$-positive definite.
\end{proof}

\begin{proposition}\label{prop:Min is Min for Symmetrization}
For some kernel $Y: (\mathbb S^{d-1})^k \rightarrow \mathbb{R}$ let $K$ be defined by
$$K( x_1,\ldots, x_k) = \frac{1}{k!} \sum_{\pi} Y( x_{\pi(1)}, \ldots, x_{\pi(k)}),$$
where $\pi$ varies over all permutations of $\{1, \ldots, k\}$. Then $K$ is a symmetric kernel, and $I_K$ is minimized by $\mu \in \mathcal{P}(\mathbb S^{d-1})$ if and only if $I_Y$ is as well.
\end{proposition}

The proof is identical to that of Lemma \ref{lem:kPosDef to nPosDef}. We note that, unlike in Lemma \ref{lem:kPosDef to nPosDef}, $k$-positive definiteness of $Y$ does not imply that $K$ is also $k$-positive definite. In fact, in the three-input case, $-V^2$ and $-A^2$ in this paper are examples of symmetric kernels that are not $3$-positive definite \cite[Propositions 6.9 and 6.10]{BFGMPV1} but are the symmetrizations of $3$-positive definite kernels (modulo a constant), see \eqref{eq:Volume^2 SDP decomposition} and \eqref{eq:Area^2 SDP Decomposition}. 

We finally remark that the discussion of this section generalizes to arbitrary compact metric spaces in place of  the sphere $\mathbb S^{d-1}$. 

\subsection{Two-input volumes}\label{sec:k=2}

Here, we address the two-input versions of $V^2$ and $A^2$ on the sphere.

\begin{proposition}\label{prop:vsas}
Let $k=2$. On the sphere $\mathbb{S}^{d-1}$, $\sigma$ is a maximizer of the two-input energies $I_{A^s}$ and $I_{V^s}$ for $0<s<2$. Moreover, in the case of $A^s$, $\sigma$ is the unique maximizer.
\end{proposition}

\begin{proof}
It is well known, see e.g. \cite[Proposition 2.3]{BGMPV1}, that  $\sigma$ is a minimizer of $I_K$, where $K(x,y)=f(\langle x,y\rangle)$, if and only if for the Gegenbauer expansion $f(t)=\sum_{m=0}^{\infty} \hat{f}_m P_m^d(t)$, $\hat{f}_m\geq 0$ for all $m\geq 1$ (which, according to Theorem \ref{thm:sch}, is equivalent to the  fact that $f$ is positive definite on $\mathbb S^{d-1}$ modulo an additive constant). Moreover, $\sigma$ is the unique minimizer if $\hat{f}_m> 0$ for all $m\geq 1$.

We note that it is sufficient to use a weaker condition based on Maclaurin expansions of $f$. Assume $f(t)=\sum_{m=0}^{\infty} f^*_m t^m$ for $t\in[-1,1]$, with the series converging uniformly and absolutely. Each function $t^m$ is positive definite on $\mathbb{S}^{d-1}$ by Schur's product theorem, 
and, by Theorem \ref{thm:sch}, it can be represented as a non-negative combination of Gegenbauer polynomials with, in particular, a positive coefficient for $P_m^d(t)$. This means that, whenever all $f^*_m$ are non-negative (positive) for $m\geq 1$, then all Gegenbauer coefficients $\hat{f}_m$ for $m\geq 1$ are also non-negative (positive). We need  to show that $\sigma$ is a maximizer, so it is sufficient to check that all coefficients, starting from $m=1$, of the  Maclaurin expansions of $V^s$ and $A^s$ are nonpositive.

Indeed, 
\begin{equation*}
V^s(x,y) = (V^2)^{s/2} = \left( 1-\langle x,y\rangle^2 \right)^{s/2} =  \sum_{m=0}^{\infty} (-1)^m \binom{s/2}{m} \langle x,y \rangle^{2m}.
\end{equation*}
Similarly, 
\begin{equation*}
A^s(x,y) = (A^2)^{s/2} = (2-2\langle x,y\rangle)^{s/2}=2^{s/2} (1-\langle x,y\rangle)^{s/2} = 2^{s/2}\sum_{m=0}^{\infty} (-1)^m \binom{s/2}{m} \langle x,y \rangle^m.
\end{equation*}
In both cases, $(-1)^m \binom{s/2}{m}$ is negative for all $m\geq 1$ so $\sigma$ is a maximizer for $V^s$ and the unique maximizer for $A^s$.
\end{proof}

\begin{remark}
Since $V$ is invariant under  central symmetry, it would be natural  to consider it as a potential on the projective space  $\mathbb{RP}^{d-1}$. Under this setup the uniform distribution over $\mathbb{RP}^{d-1}$ is the unique maximizer of $I_{V^s}$.
\end{remark}

Since $A(x,y) = \| x-y \|$, the statements  about $A^s$  in Proposition \ref{prop:vsas} can be viewed as a special case of a more general result of  Bjorck  \cite{Bj}. Below we collect his results specialized to the sphere.

\begin{proposition}[Bjorck \cite{Bj}]\label{prop:k=2}
Let $k=2$, i.e. $A (x,y) = \| x- y \|$. For the two-input energy $I_{A^s}$ on the sphere $\mathbb{S}^{d-1}$,
\begin{itemize}
\item if $0<s<2$, then $\sigma$ is the unique maximizer of $I_{A^s}$;
\item if $s = 2$, then $\mu$ is a maximizer of $I_{A^s}$ if and only if $\mu$ is balanced;
\item if $s > 2$, then the maximizers of $I_{A^s}$ are exactly measures of the the form $\frac{1}{2}( \delta_{p} + \delta_{-p})$, for some $p \in \mathbb{S}^{d-1}$.
\end{itemize}
\end{proposition}
A similar proposition about the minimizers over $\mathcal P(\mathbb S^{d-1})$ can be formulated for powers of $V$ in the two-input case. 

\begin{proposition}\label{prop:Vk=2}
Let $k=2$, i.e. $V (x,y) = \left( 1-\langle x,y\rangle^2 \right)^{1/2} $. For the two-input energy $I_{V^s}$ on the sphere $\mathbb{S}^{d-1}$,
\begin{itemize}
\item if $0<s<2$, then $\sigma$ is  a maximizer of $I_{V^s}$;
\item if $s = 2$, then $\mu$ is a maximizer of $I_{V^s}$ if and only if $\mu$ is isotropic;
\item if $s > 2$, then the only  maximizers (up to central symmetry and rotation) of $I_{V^s}$ are uniform measures on the elements of an orthonormal basis of $\mathbb R^d$, i.e. measures of the  the form $\frac{1}{d} \sum_{i=1}^d \delta_{e_i}$, where $\{e_i\}_{i=1}^d$ is an orthonormal basis of $\mathbb R^d$. 
\end{itemize}
\end{proposition}

The case $0<s<2$ is covered in Proposition \ref{prop:vsas} above. The phase transition case  $s=2$ follows from the case of equality in Lemma \ref{lem:frame}. The case $s>2$ can be easily handled by the linear programming method, but we give the proof of a more general statement for all $2\le k\le d$ in Corollary \ref{cor:V^s for s>2}. Exposition on the logarithmic and singular energies ($s < 0$) can be found in \cite{BHS} (and the references therein) for $A^s$ and \cite{CHS} for $V^s$.

\subsection{Comparison of two-input and multi-input energies}\label{sec:compare}  The multi-input, i.e. $k\geq 3$, generalizations of Propositions \ref{prop:k=2} and \ref{prop:Vk=2}, which are naturally more complicated, require different methods and form the main purpose of this paper.   As stated in the introduction, we believe that the uniform measure $\sigma$ still maximizes both $I_{A^s}$ and $I_{V^s}$ in the range $0<s<2$ for $k\ge 3$, but this remains a conjecture. 

When $s=2$ and $k\ge 3$, maximizers of $I_{V^2}$ are, as in Proposition \ref{prop:Vk=2},
 exactly the isotropic measures on $\mathbb S^{d-1}$ \cite{CC} (see Theorem \ref{thm:volume-gen}). However, we shall show  (see Theorem \ref{thm:area-gen}, as well as Theorems \ref{thm: triangle area squared max}  and  \ref{thm:A^2maxGen})  that the maximizers of $I_{A^2}$ for $3\le k \le d+1$ are exactly \emph{balanced isotropic measures} (and not just balanced as in Proposition \ref{prop:k=2} for $k=2$). 

 The case  $s>2$ of Proposition \ref{prop:Vk=2} for $I_{V^s}$ still holds for all $2\le k \le d$ (Corollary \ref{cor:V^s for s>2}). However, we are only able to prove an analogue of this case for $A^s$ when $k=d+1$ (Corollary \ref{cor:A^s for s>2}): the uniform measure on the vertices of a regular simplex replaces the two poles as the unique (up to rotations) maximizer of $I_{A^s}$ for $s>2$. We conjecture that maximizers of $I_{A^s}$ with $s>2$ are discrete for all $3\le k \le d+1$, but their exact structure remains elusive (see end of Section~\ref{sec:A^2 on Sphere}).

This discussion shows that in the multi-input case $k\ge 3$ the behavior of $A^s$ is significantly more complicated than that of $V^s$, which is evidenced already by the fact that the polynomial representation of $A^2$ (Lemma \ref{lem:A-formula}) is  more involved than that of $V^2$. 

\section{Discrete Energies and Optimality of the Regular Simplex}\label{sec:Discrete A and V}

Before presenting the study of maximizers of continuous energies with kernels  $V^s$ and $A^s$, we discuss their discrete analogues with $N=d+1$ points, and find that the  regular simplex is a maximizer for $0 < s < 2$. Consequently, we discover a new geometrically optimal property of the regular simplex. These statements use the results from Sections \ref{sec:vol} and \ref{sec:A^2 on Sphere}  about continuous $k$-point energies  as a tool. We chose to open with these discrete results since, in our opinion, they yield particularly elegant applications of the  theory. We start with a general statement:

\begin{theorem}\label{thm: simplex best, general}
Let $2 \leq k \leq d+1$ and $B: ( \mathbb{S}^{d-1})^k \rightarrow [0, \infty)$ be a polynomial kernel of degree at most two in each variable, such that $\sigma$ maximizes $I_B$ and whenever $x_i = x_j$ for some $i \neq j$, then $B(x_1, x_2, \ldots, x_k) = 0$.

Let $f: [0, \infty) \rightarrow \mathbb{R}$ be concave, increasing, and such that $f(0) = 0$, and define the kernel $K(x_1, \ldots, x_k) = f(B(x_1, \ldots, x_k))$. If $N= d+1$, then the set of  vertices of a regular $(N-1)$-simplex inscribed in $\mathbb{S}^{d-1}$ maximizes the discrete energy $E_K(\omega_N)$ over all $N$-point configurations on the sphere.

Moreover, if $f$ is strictly concave and strictly increasing, then the vertices of regular $(N-1)$-simplices are the only maximizers of the energy   
(if $B$ doesn't contain terms which are linear in some of the variables, the uniqueness is up to changing any individual vertex $x$ to its opposite $-x$).
\end{theorem}

\begin{proof}

Let $\omega_N = \{ z_1, \ldots, z_N\}$ be an arbitrary point configuration on $\mathbb{S}^{d-1}$. Since $B$ is zero if two of its inputs are the same, we can restrict the sum to $k$-tuples with distinct entries. Combining this with the fact that $f$ is increasing and concave, using Jensen's inequality,  we have
\begin{align*}
E_K(\omega_N) & := \frac{1}{N^k} \sum_{z_1, \ldots, z_k \in \omega_N}  f(B(z_1, \ldots, z_k)) \\
& \leq \frac{N (N-1) \cdots (N-k+1)}{N^k} f \left( \sum_{\substack{z_{j_1}, \ldots, z_{j_k} \in \omega_N \\
j_1, \ldots, j_k \text{ distinct}}} \frac{B(z_{j_1}, \ldots, z_{j_k})}{N (N-1) \cdots (N-k+1)}\right)\\
& = \frac{N (N-1) \cdots (N-k+1)}{N^k} f \left(  \frac{ N^k E_B (\omega_N)}{N (N-1) \cdots (N-k+1)}\right)\\
& \leq \frac{N (N-1) \cdots (N-k+1)}{N^k} f \left(  \frac{ N^k I_B(\sigma)}{N (N-1) \cdots (N-k+1)}\right).\\
\end{align*}
The first inequality becomes an equality if
\begin{equation*}
B(y_1,\ldots, y_k) = \frac{N^k E_B (\omega_N)}{N (N-1) \cdots (N-k+1)}
\end{equation*}
for all distinct $y_1, \ldots, y_k \in \omega_N$, while the second becomes an equality if the point configuration is a spherical $2$-design, in particular, if $\omega_N$ is a regular simplex. The case of  uniqueness is similar.  
\end{proof}

This generalizes some known results for $B(x,y) = \| x-y\|^2$ \cite{Y, CK}.  Note that this proof also extends  to provide an upper bound of the energy $E_{f \circ B}(\omega_N)$ for every $N \geq k$, and that this upper bound is achieved whenever $B(z_{j_1}, ..., z_{j_k})$ is constant for every $k$-tuple of distinct points, meaning that one may find additional optimizers of the energy. For instance, for $B=V^2$ and $N = d$, any orthonormal basis would be a maximizer (though this would not work for $B = A^2$, since an orthonormal basis is not balanced). An upper bound of this form was given for $E_{V}(\omega_N)$ in \cite[Corollary 5.2]{CC}.

We will show in subsequent  sections that $\sigma$  maximizes the continuous energies with kernels  $V^2$ and $A^2$ (Theorems \ref{thm:area-gen} and \ref{thm:volume-gen}), both of which are polynomials of degree two. Hence  Theorem \ref{thm: simplex best, general} applies,  immediately yielding  the following corollary:

\begin{corollary}\label{cor:Simplex Best A, V}
Assume that either  $K(x_1, \ldots, x_k) = V(x_1, \ldots, x_k)^s$ with   $2 \leq k \leq d$, or $K(x_1, \ldots, x_k) = A(x_1, \ldots, x_k)^s$  with  $2 \leq k \leq d+1$.

Let  $0 < s \le 2$. For $N=d+1$ points,  the discrete $k$-input  energy $E_K$  on the  sphere $\mathbb S^{d-1}$ is uniquely (up to rotations, and up to central symmetry in the case of $V^2$) maximized by the vertices of a regular simplex inscribed in $\mathbb{S}^{d-1}$.
\end{corollary}

\begin{proof}
For $0<s<2$, the function $f(t) = t^{s/2}$ is strictly concave and strictly increasing, so the Theorem immediately applies. The uniqueness in the case $s=2$ needs a separate discussion. By Theorems \ref{thm:area-gen} and \ref{thm:volume-gen}, $I_{V^2}$ is maximized by isotropic measures on $\mathbb S^{d-1}$, and $I_{A^2}$ -- by balanced isotropic measures. In the discrete case, isotropic measures on $\mathbb S^{d-1}$ are exactly unit norm tight frames. The only tight frames  on $\mathbb S^{d-1}$  with $N=d+1$ elements (up to central symmetry and rotations) are the vertices of a regular simplex \cite[Theorem 2.6]{GK}. 
\end{proof}

Taking $K = A^s$ in Corollary \ref{cor:Simplex Best A, V},  and setting $j=k-1$, we obtain an interesting  geometric result: 

\begin{corollary}\label{cor:Geometric Simplex is best}
Let $1 \leq j \leq d$, $0 < s \leq 2$, $S$ be a $d$-simplex inscribed in $\mathbb{S}^{d-1}$, $\mathcal{F}_{j}$ the set of $j$-dimensional faces of $S$, and $\mathrm{Vol}_{j}(C)$ the $j$-dimensional volume of a set $C$. Then
\begin{equation}\label{eq:T-functional}
\sum_{F \in \mathcal{F}_{j}} \mathrm{Vol}_{j}(F)^s,
\end{equation}
achieves its maximum if and only if $S$ is a regular simplex.
\end{corollary}

In the case $s = 1$, this generalizes the known results for $j=1$, i.e. the sum of distances between vertices \cite{F1}, $j=d-1$, i.e. the surface area \cite{Ta2}, and $j=d$, i.e. the volume \cite{Jo, Ta1, Ball2, HL}. We also note that \eqref{eq:T-functional} is a special case of the $T$-functional, which has received a fair amount of study, mostly in Stochastic Geometry (see e.g. \cite{A, GKT, HoLe, HMR, KMTT, KTT}).

\begin{remark}
We note that by adjusting the definition of the discrete energy $E_K$ to only include summands where all inputs are distinct, we can study lower-semicontinuous kernels $K: \Big(\mathbb{S}^{d-1} \Big)^k \rightarrow ( - \infty, \infty]$. In this case, if we define $f$ in the statement of Theorem \ref{thm: simplex best, general} as a decreasing, convex function $f: (0, \infty) \rightarrow \mathbb{R}$, with $f(0) = \lim_{x \rightarrow 0^+} f(x)$, an identical proof shows that the vertices of a regular simplex minimize $E_{f \circ B}$, as a generalization of \cite[Theorem 1.2]{CK}. In particular, as an extension of Corollary \ref{cor:Simplex Best A, V}, this shows that regular simplices are optimal for $-\log(A)$ and $-\log(V)$, as well as $A^{s}$ and $V^{s}$ for $s < 0$.
\end{remark}

\section{Semidefinite Programming and Three-input Volumes}\label{sec:SD}
{In this section we recall the basics of semidefinite programming and apply it to the maximization of integral functionals with three-input kernels. It will be shown that isotropic probability measures on $ \mathbb S^{d-1} $ maximize $ I_{V^2} $, where  $ V $ is  the $3$-volume, among all probability measures, while for the $2$-area energy integral $ I_{A^2} $, maximizers are isotropic and balanced probability measures. In Sections} \ref{sec:Powers of V} and \ref{sec:A^2 on Sphere}, {these results are generalized to larger numbers of inputs and to measures on $\mathbb{R}^d$.}

For brevity, we denote $u=\langle y,z \rangle$, $v=\langle v,z \rangle$, $t=\langle x,y \rangle$. We also take $\sigma$, as before, to be the uniform probability measure on the sphere $\mathbb{S}^{d-1}$. In \cite{BV}, Bachoc and Vallentin produced a class of infinite matrices and associated polynomials of the form
\begin{equation}\label{eq:SemDefProgYs}
(Y_{m}^d)_{i+1,j+1} (x,y,z) := Y_{m,i,j}^d(x,y,z) := P_i^{d + 2m}(u) P_{j}^{d +  2m}(v) Q_m^d(u,v,t),
\end{equation}
where $m,i,j \in \mathbb{N}_0$, $P_m^h$ is the normalized Gegenbauer polynomial of degree $m$ on $\mathbb{S}^{h-1}$ and
\begin{equation}\label{eq:BachocValQKernel}
Q_m^d(u,v,t) = ((1-u^2)(1-v^2))^{\frac{m}{2}}P_{m}^{d-1}\left(\frac{t-uv}{\sqrt{(1-u^2)(1-v^2)}}\right).
\end{equation}

\begin{remark}
Polynomials $Y_{m,i,j}^d$ in \cite{BV} were defined with certain coefficients which we omit here for the sake of simplicity.
\end{remark}

Here we provide the upper left $3\times 3$, $2 \times 2$, and $1 \times 1$ submatrices of infinite matrices $Y_0^d$, $Y_1^d$, and $Y_2^d$, respectively, which is all that we need for the rest of this section:
$$\begin{pmatrix} 
1 & v & \frac {dv^2-1} {d-1}\\
u & uv & u \frac {dv^2-1} {d-1}\\ 
\frac {dv^2-1} {d-1} & \frac {du^2-1} {d-1} v & \frac {du^2-1} {d-1} \frac {dv^2-1} {d-1}\end{pmatrix}$$
$$ \begin{pmatrix} t-uv & u (t-uv)\\ 
v(t-uv) & uv(t-uv)\end{pmatrix}, \begin{pmatrix} \frac {(d-1)(t-uv)^2-(1-u^2)(1-v^2)} {d-2}\end{pmatrix}.$$
By letting $\pi$ run through the group of all permutation of the variables $x$, $y$, and $z$, and averaging, they defined the following symmetric matrices and associated polynomials
\begin{equation*}
(S_{m}^d)_{i+1,j+1} (x,y,z) := S_{m,i,j}^d (x,y,z) := \frac{1}{6} \sum_{\pi} Y_{m,i,j}^d (\pi(x), \pi(y), \pi(z)).
\end{equation*}
These polynomials and matrices have a variety of nice properties: 
\begin{enumerate}
\item\label{i} For any $\mu \in \mathcal{P}(\mathbb{S}^{d-1})$ and $e \in \mathbb{S}^{d-1}$,
$$ \int_{\mathbb{S}^{d-1}} \int_{\mathbb{S}^{d-1}} Y_{m}^d(x,y,e)\, d\mu(x) d \mu(y)$$
and
$$ S_m^d(\mu) := \int_{\mathbb{S}^{d-1}} \int_{\mathbb{S}^{d-1}} \int_{\mathbb{S}^{d-1}} S_{m}^d(x,y,z)\, d\mu(x) d \mu(y) d \mu(z)$$
are positive semidefinite, i.e. all principal minors (formed by finite submatrices) are nonnegative. 
\item\label{ii} For $(m,i,j) \neq (0,0,0)$, $I_{S_{m,i,j}^d}(\sigma) = 0$, and for all $e \in \mathbb{S}^{d-1}$, $ I_{Y_{m,i,j}^d}(\sigma, \sigma, \delta_{e}) = 0$.
\item\label{iii} For $m \geq 1$ and $e \in \mathbb{S}^{d-1}$, $I_{S_{m,i,j}^d}(\delta_e) = I_{Y_{m,i,j}^d}(\delta_e) = 0$.
\end{enumerate}
We note that the paper  \cite{BV} was only concerned with finite point sets. However, the results naturally extend to the continuous setting, and \eqref{i} is simply the extension of Corollary 3.5 in \cite{BV}, while  \eqref{ii} follows from the construction of $Y_{m,i,j}$'s from spherical harmonics (see Theorem 3.1 and the preceding text, as well as equation (11), in \cite{BV}). Finally, \eqref{iii} follows from the fact that $Q_m^d(1,1,1) = 0$.

Now consider an infinite, symmetric, positive semidefinite matrix $A$ with finitely many nonzero entries. Then for any $m \geq 1$ and $\mu \in \mathcal{P}( \mathbb{S}^{d-1})$, $\Tr( S_m^d(\mu) A) \geq 0$, with equality if $\mu = \sigma$.  (Indeed, observe that,  for two matrices positive definite matrices $B=(b_{ij})$, $C = (c_{ij})$, Schur's theorem implies that the Hadamard product $B\circ C = (b_{ij} c_{ij})$ is positive definite, which leads to the inequality  $\Tr (BC) = \sum_{i,j} b_{ij} c_{ij} \ge 0$.) 

Likewise, let $A_0$ be an infinite, symmetric, positive semidefinite matrix $A$ with finitely many nonzero entries and such that all entries in the first row and first column are zeros. Then for any probability measure $\mu$, $\Tr( S_m^d(\mu) A_0) \geq 0$, with equality if $\mu = \sigma$. In this case, we require zeros in the first row and column due to the fact that $S_{0,0,0}^d$ is a constant, so we would not get equality for $\sigma$ in the above inequality. This gives us the following:
\begin{theorem}\label{thm:SemiDefMin}
Let $n \in \mathbb{N}_0$. For each $m \leq n$, let $A_m$ be an infinite, symmetric, positive semidefinite matrix with finitely many nonzero entries, with the additional requirement that $A_0$ has only zeros in its first row and first column. Let
$$K(x,y,z) = \sum_{m=0}^{n} \Tr( S_{m}^d(x,y,z)\, A_m).$$
Then $\sigma$ is a minimizer of $I_K$ over probability measures on the sphere $\mathbb{S}^{d-1}$.
\end{theorem}
\noindent Naturally, adding a constant to $K$  does not change this statement, and multiplying by $-1$ turns  it into a maximization result.  
Observe also that, when $A$ is a diagonal matrix, $\Tr( S_m^d A) $ is simply a positive linear combination of the diagonal elements of $ S_m^d $.  Theorem \ref{thm:SemiDefMin} is often applied in this way (see the proofs of Theorems \ref{thm:vol squared max} and \ref{thm: triangle area squared max} below), in close  analogy to Theorem \ref{thm:sch}.

\subsection{Volume of a parallelepiped}

Maximizing the sum of distances between points on a space (or the corresponding distance integrals)  is a very natural optimization problem for two-input kernels, and one which has garnered a fair amount of attention (see, e.g. \cite{B, BBS, BS, BD, BDM, Bj, F1, F2, Sk, St}), 
and, as mentioned in the introduction,  higher dimensional analogues, such as area and volume, yield natural extensions for kernels with more inputs. In this section we discuss such questions for $k=3$ inputs, focusing on volume squared and area squared, as these produce polynomials which are easier to work with.

We first consider the kernel
\begin{equation*}
K(x,y,z) = V^2(x,y,z) = \det\begin{pmatrix} 1 & u & v \\ u & 1 & t \\ v & t & 1 \end{pmatrix}=1-u^2-v^2-t^2+2uvt
\end{equation*}
on the sphere $\mathbb{S}^{d-1}$, with $d > 2$, and where $V(x,y,z)$ is the volume of the parallelepiped formed by the vectors $x$, $y$, and $z$.  As mentioned in \cite{BFGMPV1}, $-V^2$ is not 3-positive definite (modulo a constant) but as we show here, $\sigma$ is a minimizer of $I_{-V^2}$, i.e. a maximizer of $I_{V^2}$.

Indeed, we see that
\begin{equation}\label{eq:Volume^2 SDP decomposition}
V^2(x,y,z) = \frac{(d-1)(d-2)}{d^2} - \frac{(d-1)(d-2)}{d^2} S_{0,2,2}^d - \frac{4(d-2)}{d} S_{1,1,1}^d - \frac{(3d-4)(d-2)}{d (d-1)} S_{2,0,0}^d,
\end{equation}
so Theorem \ref{thm:SemiDefMin} tells us that $\sigma$ is a maximizer.
Moreover, since $V^2$ is a polynomial of degree two in every variable, and has no linear terms, any isotropic measure on the sphere is also a maximizer, and in fact this classifies all maximizers. 
\begin{theorem}\label{thm:vol squared max}
Isotropic probability measures on the sphere maximize $I_{V^2}$ over $\mathcal{P}(\mathbb{S}^{d-1})$.
\end{theorem}

\subsection{Area of a triangle}
Using the same method as in Theorem \ref{thm:vol squared max}, we can show that $\sigma$ is a maximizer of $I_{A^2}$, where $A(x,y,z)$ is the area of a triangle, since
\begin{equation}\label{eq:Area^2 SDP decomposition}
A^2(x,y,z) = \frac{1}{4} \Big(3 \frac{d-1}{d} - 3 \frac{d-2}{d-1} S_{2,0,0}^d -6 S_{1,1,1}^d - 6S_{1,0,0}^d - 3 \frac{d-1}{d}S_{0,2,2}^d \Big).
\end{equation}

However, we can also prove this with a slightly different method, which acts as a special case of Theorem \ref{thm:vol squared max}, a more general statement that we will prove by means of $k$-point bounds.

\begin{theorem}\label{thm: triangle area squared max}
Suppose $d \geq 2$, and let $ A^2(x,y,z) $ be the square of the area of the triangle  with vertices at  $x$, $y$, $z\in \mathbb S^{d-1}$. Then the uniform surface measure $\sigma$ maximizes $I_{A^2} (\mu)$ over $\mathcal{P} (\mathbb{S}^{d-1})$. Moreover, any balanced, isotropic measure $\mu \in \mathcal{P} (\mathbb{S}^{d-1})$ maximizes $I_{A^2}$.
\end{theorem}

\begin{proof}
Using Heron's formula, we express $A^2(x,y,z)$ via the scalar products of $x,y,z$:
\begin{equation}\label{eq:Area^2 SDP Decomposition}
A^2 (x,y,z)   =   \frac34  -\frac12 (u+v+t) + \frac12 (uv + vt +tu) - \frac{1}4 ( u^2 + v^2 + t^2) = \frac {3} {4} - \frac 3 2 S_{1,0,0}^d - \frac{1}4 ( u^2 + v^2 + t^2).
\end{equation}

Note that $\sigma$ minimizes both $S_{1,0,0}^d$ and $u^2+v^2+t^2$ by Theorem \ref{thm:SemiDefMin} and Lemma \ref{lem:frame}, respectively. Therefore, for all $\mu \in \mathcal{P}(\mu)$,
$$I_{A^2}(\mu)\leq \frac 3 4 - \frac 1 4 \cdot \frac 3 d = \frac {3(d-1)} {4d} = I_{A^2}(\sigma).$$
More generally, maximizers of $I_{A^2}$ must be isotropic measures in order to achieve the sharp bound from Lemma \ref{lem:frame} and must be balanced so that $S_{1,0,0}^d$ vanishes on them.
\end{proof}

An alternative proof of this result is also given in \cite[Theorem 6.7]{BFGMPV1}. We also would like to remark that since
\begin{equation*}
3 \frac{d-2}{d-1} S_{2,0,0}^d +6 S_{1,1,1}^d  + 3 \frac{d-1}{d}S_{0,2,2}^d + \frac{3}{d} = u^2 + v^2 + t^2,
\end{equation*}
which follows from \cite[Proposition 3.6]{BV}, Lemma \ref{lem:frame} follows from Theorem \ref{thm:SemiDefMin}, which demonstrates an instance of obtaining  $2$-point bounds from $3$-point bounds.

\section{Maximizing $k$-volumes}\label{sec:vol}

{
    This section collects results on maximization of volume integral functionals over probability measures with unit second moment in $ \mathbb R^d $, denoted as before by $ \mathcal P^*(\mathbb R^d) $, for $k\ge 3$ inputs. In some cases we will further restrict the supports of such measures to the unit sphere, thereby optimizing over $ \mathcal P(\mathbb S^{d-1}) $. 
   As  in the rest of the paper, we are interested in powers of  two kernels:  the $k$-dimensional Euclidean volume of the parallelepiped   
   $V(x_1,\ldots,x_k)$,  
   and the $(k-1)$-dimensional 
   volume of the simplex 
   $A (x_1,\ldots,x_k)$.}


\subsection{Maximizing the powers of  $V$}\label{sec:Powers of V}

As in the previous section, we start with  $ V^2 $, i.e.  the squared $k$-dimensional volume of a parallelepiped spanned by the vectors $x_1, \ldots, x_k$, equal to the determinant of the Gram matrix of the set of vectors $\{x_1,\ldots,x_k\}\subset \mathbb R^d$. Alternatively, $\frac 1 {(k!)^2} V^2(x_1, \ldots, x_k)$ can be seen as the square of the Euclidean volume of the simplex with vertices $0, x_1, \ldots, x_k$.  The following theorem (in a slightly different form) can be found in the literature: 

\begin{theorem}\label{thm:volume-gen} Let $d\ge 3 $ and $3\le k \le d$. 
The  set of maximizing measures of $I_{V^2}$ in $\mathcal{P}^*(\mathbb{R}^d)$ is the set of isotropic measures on $\mathbb{R}^d$. The value of the maximum is $\frac {k!}{d^k}\binom{d}{k}$.

As a corollary,  isotropic measures on $\mathbb S^{d-1}$ (which include the uniform surface measure $\sigma$) are exactly the  maximizers of  $I_{V^2}$ over $\mathcal{P}(\mathbb S^{d-1})$. 
\end{theorem}

This theorem was proved by Rankin \cite{R} for $k=d$ and by Cahill and Casazza \cite{CC} in the general case. In both papers, the statements are for finite spherical sets but the proofs work for measures in $\mathcal{P}^*(\mathbb{R}^d)$ with only minor adjustments. We also note that due to Remark \ref{rem:proj}, it is sufficient to prove the result for the spherical case only, since $V^2$ is homogeneous of degree two. The equality case of Theorem \ref{thm:volume-gen} was also treated in a more general context in \cite{FNZ, P}.

{Theorem \ref{thm:volume-gen} allows one to characterize the minimizers of $ I_{V^s} $ with $ s>2 $ as well.}

\begin{corollary}\label{cor:V^s for s>2}
For $s > 2$, the energy $I_{V^s}$  on $\mathbb{S}^{d-1}$  is uniquely (up to rotations and central symmetry) maximized by the uniform measure on an orthonormal basis. 
\end{corollary}

\begin{proof}
For $s > 2$, $V^2(x_1, \ldots, x_k) \geq V^s(x_1, \ldots, x_k)$ for all $x_1, \ldots, x_k \in \mathbb{S}^{d-1}$, with equality exactly when $x_1, \ldots, x_k$ is an orthonormal set (so the volume is $1$) or when $x_1, \ldots, x_k$ are linearly dependent (so the volume is $0$). Thus, for all $\mu \in \mathcal{P}(\mathbb{S}^{d-1})$,
\begin{equation}
\frac {k!}{d^k}\binom{d}{k} \geq I_{V^2}(\mu) \geq I_{V^s}(\mu).
\end{equation}
The first inequality becomes an equality if $\mu$ is isotropic, and the second inequality becomes equality if any $x_1, \ldots, x_k \in \operatorname{supp}(\mu)$ are either orthonormal or linearly dependent. Since the support of an isotropic measure must be full-dimensional, both of these conditions  occur simultaneously  if and only if $\mu(\{- e_j, e_j\}) = \frac{1}{d}$ for $j = 1, \ldots, d$ for some orthonormal basis $e_1, \ldots, e_d$.
\end{proof}

It is easy to see that the uniform distribution on an orthonormal basis is not a maximizer of $V^s$ for $0 < s < 2$.

\subsection{Maximizing the powers of $A$}\label{sec:Powers of A}

We now turn to the powers of  $A (x_1,\ldots,x_k)$, the $(k-1)$-dimensional  volume of the simplex with vertices $x_1, \ldots, x_k$,  and again start by considering the kernel $A^2$. 
   The result of Theorem \ref{thm:volume-gen} for $V^2$ can be used to obtain a similar statement for the measures in $\mathcal{P}^*(\mathbb{R}^d)$ maximizing $ I_{A^2} $, in the case of  $k= d+1$ inputs, i.e when the simplex is full-dimensional. The main idea is to embed $ \mathbb R^d $ into $ \mathbb R^{d+1} $, treat the value of $ A^2(x_1,\ldots,x_{d+1}) $ as the value of $ V^2(y_1,\ldots,y_{d+1}) $ for a suitable $ y_1,..., y_k $, and then use Theorem \ref{thm:volume-gen}. 
   
   We shall defer the calculation of the maximal value of $ I_{A^2} $ until Theorem~\ref{thm:area-gen}, which also gives an alternative proof of the characterization of maximizers on the sphere $\mathbb S^{d-1}$, moreover, addressing the case of {\emph{any}} number of inputs $3\le k \le d+1$, rather than just $k=d+1$ as in the theorem below. However, the result of this section, Theorem \ref{thm:A^2maxGen}, applies to measures on $\mathbb R^d$, while Theorem~\ref{thm:area-gen} is restricted to the sphere. 

\begin{theorem}\label{thm:A^2maxGen}
For $d \geq 2$ and $k = d+1$, maximizers of $I_{A^2}$ in $\mathcal{P}^*(\mathbb{R}^d)$ are the balanced, isotropic probability measures on $\mathbb{R}^d$.
\end{theorem}

\begin{proof}
In this proof, we say that a measure $ \mu $ on $ \mathbb R^d $ is $ d $-isotropic if equation \eqref{eq:isotropic_init} holds. 
Given a unit basis vector $ e_{d+1}\in \mathbb R^{d+1} $, we identify $ \mathbb R^d $ with the hyperplane in $ \mathbb R^{d+1} $, orthogonal to $ e_{d+1} $ and passing through the origin.

To reduce the value of $ I_{A^2} $ to that of $ I_{V^2} $, given a measure $ \mu $ on $ \mathbb R^d $, denote its pushforward to $\mathbb R^{d+1}$ by $\hat\mu$:
\begin{equation}
    \label{eq:muhat}
    \hat\mu:= \psi_\# \mu,
\end{equation}
where the map $ \psi: \mathbb R^d \to \mathbb R^{d+1} $ is
\[
    \psi(x) :=  \sqrt{\frac{d}{d+1}}x + \frac1{\sqrt{d+1}}e_{d+1}.
\]
It is understood here that $ x\in\mathbb R^d \subset \mathbb R^{d+1} $, so the addition in right-hand side is in \(\mathbb R^{d+1}\).  

Recall that a $ (d+1) $-dimensional simplex with base of $ d $-dimensional volume $ S $ and height $ h $ has $ (d+1) $-dimensional volume  ${ S h}/{(d+1)} $. This gives,  with $V^2 = V^2(x_1,\ldots,x_{d+1}) $ the square of the $ (d+1) $-dimensional volume of the parallelepiped spanned by its inputs,
\begin{equation}\label{eq:AtoV}
I_{V^2}(\hat\mu) = (d!)^2 \frac{d^{d}}{(d+1)^{d+1}}I_{A^2}(\mu),
\end{equation}
where we account for the fact that $ V $ includes a factor of $ (d+1)! $, whereas $ A $ does not. 

By Theorem \ref{thm:volume-gen}, the functional $I_{V^2}$ on the left-hand side of equation~\eqref{eq:AtoV} is maximized over $\mathcal{P}^*(\mathbb{R}^{d+1})$ exactly when $\hat{\mu}$ is isotropic. To finish the proof, it remains to observe that the pushforward $ \hat\mu = \psi_\#\mu $ is $ (d+1)$-isotropic if and only if $ \mu $ is $ d $-isotropic and balanced. Indeed, writing $ x = (x^{(1)}, \ldots, x^{(d+1)}) = (y, x^{(d+1)}) $, we have for $ 1\leq i \leq j \leq d+1 $
\begin{equation}
    \label{eq:isotropyGamma}
    \int_{\mathbb{R}^{d+1}} x^{(i)}x^{(j)}\, d\hat\mu(x) =
    \begin{cases}
        \frac{d}{d+1} \int_{\mathbb{R}^d} y^{(i)}y^{(j)}\, d\mu(y) &  j \leq d,\\
        \frac{\sqrt{d }}{d+1}  \int_{\mathbb{R}^d} y^{(i)}\, d\mu(y) & i \leq d, \  j=d+1,\\
        \frac{1}{d+1} & i=j=d+1.
    \end{cases}
\end{equation}
By \eqref{eq:isotropic_init}, $ (d+1) $-isotropy of $ \hat\mu $ means the integral in the left-hand side of \eqref{eq:isotropyGamma} is equal to $ \delta_{ij}/(d+1) $, $ 1\leq i\leq j\leq d+1 $. In particular, then the first integral in the right-hand side is equal to $ \delta_{ij}/d $, which is precisely the condition for $ d $-isotropy of $ \mu $, and the integrals of $ y^{(i)} $ are all equal to zero, implying $ \mu $ is balanced. The converse implications obviously follows along the same lines. 
\end{proof}

We can use the same methods as in Corollary \ref{cor:V^s for s>2} to find the maximizers of $ I_{A^s} $, for larger powers, on the sphere $\mathbb S^{d-1}$.
\begin{corollary}\label{cor:A^s for s>2}
Let $s > 2$ and $A(x_1, \ldots, x_{d+1})$ be the $d$-dimensional volume of a simplex with vertices $x_1, \ldots, x_{d+1} \in \mathbb{S}^{d-1}$. Then $I_{A^s}$ is uniquely (up to rotations) maximized by the uniform distribution on the vertices of a regular $d$ simplex.
\end{corollary}

\begin{proof}
We know that $A(x_1, \ldots, x_{d+1})$ is maximized exactly when $x_1, \ldots, x_{d+1}$ are the vertices of a regular simplex (see, e.g. \cite{Jo, Ta1, Ball2, HL}, see also the case $j=d$ of Corollary \ref{cor:Geometric Simplex is best}). Let $\alpha$ be that maximum volume. We see that for $s > 2$,
\begin{equation*}
 A^2(x_1, \ldots, x_{d+1}) \geq \alpha^{2-s} A^s(x_1, \ldots, x_{d+1})
\end{equation*}
for all $x_1, \ldots, x_{d+1} \in \mathbb{S}^{d-1}$, with equality exactly when $A(x_1, \ldots, x_{d+1})$ is $0$ or $\alpha$.

We know that, for all $\mu \in \mathcal{P}(\mathbb{S}^{d-1})$
\begin{equation}
\frac{d+1}{d! d^d} \geq I_{A^2}(\mu) \geq \alpha^{2-s} I_{A^s}(\mu).
\end{equation}
The first inequality becomes an equality when $\mu$ is balanced and isotropic, and the second becomes an equality when $A(x_1, \ldots, x_{d+1})$ is $0$ or $\alpha$ for all $x_1, \ldots, x_{d+1} \in \operatorname{supp}(\mu)$. These both occur exactly when $\mu$ is the uniform distribution on the vertices of a regular $d$-simplex.
\end{proof}

For $0 < s < 2$, the uniform distribution on the vertices of a regular simplex is not a maximizer of $I_{A^s}$. It is also not clear which measures maximize $I_{A^s}$ for $2<s$ for general  $k< d+1$. We conjecture that the maximizers are again discrete in this case (see the discussion at the end of Section~\ref{sec:A^2 on Sphere}).

\section{Kernels for \texorpdfstring{$k$}{k}-point bounds}\label{sec:k-point}

In this section, we explain how to construct a class of continuous (in certain cases, polynomial) kernels which are $k$-positive definite and whose energy is minimized by $\sigma$. These kernels are a generalization of those developed by Bachoc and Vallentin in \cite{BV}, and similar to the kernels given in \cite{Mu} and \cite{DMOV}, all of which were used for obtaining $k$-point semidefinite programming bounds. We provide a slight alteration to these kernels, so that the inputs are no longer restricted to being linearly independent, or constrained to some proper subset of the sphere.

Consider the  points $\{x_1,\ldots, x_k\}\subset \mathbb S^{d-1}$ with $k \leq d+1$. Suppose that $x_3,\ldots, x_k$ are linearly independent and $x_1, x_2 \not\in X = \sp\{x_3, \ldots, x_k\}$, and denote the orthogonal projections of $x_1$ and $x_2$ onto $X^{\perp}$ as $y_1$ and $y_2$, respectively. Then the normalized vectors $\frac{y_1}{\|y_1\|}$ and $\frac{y_2}{\|y_2\|}$ belong to the unit sphere in the $(d-k+2)$-dimensional space $X^{\perp}$. If $k \leq d$, then on this unit sphere, the kernel given by $P_l^{d-k+2}( \langle x, y \rangle)$ is positive definite, suggesting that we may be able to build a $k$-positive definite kernel from
\begin{equation}\label{eq:SDP Building Block}
P_l^{d-k+2}\Big( \Big\langle \frac{y_1}{\|y_1\|}, \frac{y_2}{\|y_2\|} \Big\rangle \Big).
\end{equation}
If $k = d+1$, then $\frac{y_1}{\|y_1\|}, \frac{y_2}{\|y_2\|} \in \mathbb{S}^0 = \{ -1, 1 \}$, and we see that $1$ and $\frac{y_1}{\|y_1\|} \frac{y_2}{\|y_2\|}$ make a basis for positive definite functions. Of course, for $l > 0$, $P_l^{d-k+2}\Big( \Big\langle \frac{y_1}{\|y_1\|}, \frac{y_2}{\|y_2\|} \Big\rangle \Big)$ is not well defined, as a function of $x_1, \ldots, x_k$, if $x_1$ or $x_2$ is in $X$, and may not be continuous whenever the dimension of $X$ changes. We can modify \eqref{eq:SDP Building Block} to account for these issues, arriving at the following polynomial kernel. In what follows, we denote $W$ to be the Gram matrix of $x_3, \ldots, x_k$, and we set $P^1_0 =1$ and $P^1_1(t) = t$ (these are the only cases when $P^1_j$  are defined).

\begin{theorem}\label{thm:Qdef}
With the notation above, for any $l \in \mathbb{N}_0$, the function $Q_{k,l}^{d}: \Big( \mathbb{S}^{d-1} \Big)^k \rightarrow \mathbb{R}$ defined by
\begin{equation}\label{eq:Qdef}
Q_{k,l}^{d}( x_1, \ldots, x_k) = \det(W)^l  \|  y_1  \|^{l}  \| y_2 \|^{l} P_l^{d-k+2}\Big( \Big\langle \frac{y_1}{\|y_1\|}, \frac{y_2}{\|y_2\|} \Big\rangle \Big)
\end{equation}
is a rotationally-invariant $k$-positive definite polynomial kernel and $I_{Q_{k,l}^d}$ is minimized by $\sigma$.
\end{theorem}

We note that these kernels are symmetric in the last $k-2$ variables as well as   in the first two variables.

\begin{proof}
Note that $Q_{k,0}^d = 1$, so our claim holds in these cases. Now, assume that $l \in \mathbb{N}$.

Rotational-invariance follows immediately from the rotational-invariance of $W$, $\| y_1 \|$, $\|y_2\|$ and $\langle y_1, y_2 \rangle$.

In what follows, we denote $u_{i,j} = \langle x_i, x_j \rangle$ for all $i$ and $j$, and for $h \in \{ 1, 2\}$, $w_h = (u_{h,3}, \ldots, u_{h,k})^T$, $y_h^{\perp} = x_h - y_h$, and $z_h = \frac{y_h}{\|y_h\|}$.

We first must show that our kernel $Q_{k,l}^{d}$ is well-defined. Write  $y_1^{\perp} = \sum_{j=3}^{k} \alpha_j x_j$ and $y_2^{\perp} = \sum_{j=3}^{k} \beta_j x_j$, and denote $\alpha = ( \alpha_3 , \ldots, \alpha_k)^T$ and $\beta = ( \beta_3, \ldots, \beta_k)^T$. Since for $3 \leq j \leq k$,
\begin{equation*}
u_{1,j} = \langle x_1, x_j \rangle = \langle y_1^{\perp}, x_j \rangle =  \sum_{i=3}^{k} \alpha_i u_{i,j} \quad \text{ and } \quad u_{2,j} = \langle x_2, x_j \rangle = \langle y_2^{\perp}, x_j \rangle =  \sum_{i=3}^{k} \beta_i u_{i,j} 
\end{equation*}
we conclude that $w_1 = W \alpha$ and $w_2 = W \beta$.

Now assume that $x_3, \ldots, x_k$ are linearly independent. Consequently, $\alpha = W^{-1} w_1$ and $\beta = W^{-1} w_2$, so
\begin{equation}
\langle y_1^{\perp}, y_2^{\perp} \rangle = \alpha^T W \beta = w_1^T W^{-1} W W^{-1} w_2 = w_1^T W^{-1} w_2,
\end{equation}
and similarly
\begin{equation}
\langle y_1^{\perp}, y_1^{\perp} \rangle = w_1^T W^{-1} w_1 \text{ and } \langle y_2^{\perp}, y_2^{\perp} \rangle = w_2^T W^{-1} w_2.
\end{equation}
We then see that
\begin{equation}
\langle y_1, y_2 \rangle =  \langle x_1, x_2 \rangle - \langle y_1^{\perp}, y_2^{\perp} \rangle = u_{1,2} - w_1^T W^{-1} w_2,
\end{equation}
\begin{equation}
\| y_1 \|^2 = 1 - w_1^T W^{-1} w_1 \text{ and } \| y_2 \|^2 = 1 - w_2^T W^{-1} w_2.
\end{equation}
Thus, if $x_1, x_2 \not\in X$, we can rewrite \eqref{eq:Qdef} as
\begin{align}
Q_{k,l}^d (x_1, \ldots, x_k) = \det(W)^l \Big( \big(1-w_1^T W^{-1} w_1 \big) & \big(1-w_2^T W^{-1} w_2 \big) \Big)^{l/2}  \\ \nonumber &   \times P_l^{d-k+2}\left(\frac {u_{1,2} - w_1^T W^{-1} w_2} {\sqrt{1-w_1^T W^{-1} w_1}\sqrt{1- w_2^T W^{-1} w_2}} \right).
\end{align}

Letting $P_l^{d-k+2}(t) = \sum_{m=0}^{ \lfloor \frac{l}{2} \rfloor} a_{l-2m} t^{l-2m}$, we see that
\begin{align}\label{eq:Qexpansion1}
Q_{k,l}^{d}( x_1, \ldots, x_k) = \sum_{m=0}^{ \lfloor \frac{l}{2} \rfloor} a_{l-2m} & \Big( \det(W) u_{1,2} -w_1^{T} \adj(W) w_2 \Big)^{l - 2m} \; \Big(\det(W) -w_1^{T} \adj(W) w_1  \Big)^{m} \; \\ \nonumber & \times \Big(\det(W) -w_2^{T} \adj(W) w_2  \Big)^{m},
\end{align}
where $\adj(W)$ is the adjugate matrix of $W$. This is a polynomial of the inner products of $x_1,\ldots, x_k$, and so is well defined for all $x_1, \ldots, x_k \in \mathbb{S}^{d-1}$.

In addition, by rewriting \eqref{eq:Qdef} as
\begin{equation}\label{eq:Qexpansion2}
Q_{k,l}^{d}( x_1, \ldots, x_k) = \det(W)^l \sum_{m=0}^{ \lfloor \frac{l}{2} \rfloor} a_{l-2m} \langle y_1, y_2 \rangle^{l - 2m} \|y_1\|^{m} \|y_2\|^{m},
\end{equation}
for $k \leq d$ and
\begin{equation}\label{eq:Qexpansion2, k=d+1}
Q_{d+1,1}^d(x_1, \ldots, x_k) = \det(W) \langle y_1, y_2 \rangle,
\end{equation} 
we see that $Q_{k,l}^d$ is zero if $x_3, \ldots, x_k$ are linearly dependent.

If $k = d+1$ and $l =1$, \eqref{eq:Qexpansion2, k=d+1} shows us that for any fixed $x_3, \ldots, x_{d+1} \in \mathbb{S}^{d-1}$ and $\mu \in \mathcal{M}(\mathbb{S}^{d-1})$,
\begin{equation}\label{eq:I_Q k = d+1}
I_{U_{Q_{d+1,l}^d}^{x_3, \ldots, x_{d+1}}}(\mu) = \int_{\mathbb{S}^{d-1}} \int_{\mathbb{S}^{d-1}} \det(W) y_1 y_2\, d\mu(x_1) d\mu( x_2) = \det(W) \Big( \int_{\mathbb{S}^{d-1}} y_1 d \mu(x_1) \Big)^2 \geq 0,
\end{equation}
so $Q_{k, 1}^d$ is $k$-positive definite. Note that since $W$ is the Gram matrix of $x_3, \ldots, x_k$, its determinant is nonnegative. We now want to show that this energy is zero when $\mu = \sigma$. We first note that this occurs if $x_3, \ldots, x_k$ are linearly dependent, so let us assume that $x_3,\ldots, x_k$ are linearly independent, and set $f( y_1^{\perp}, y_1) = f(x_1) = y_1$. Denoting the unit ball in $\mathbb{R}^n$ as $\mathbb{B}^n$, we have, by Lemma A.5.4 of \cite{DX}, that
\begin{align*}
\int_{\mathbb{S}^{d-1}} f(x_1)\, d\sigma(x_1) & = \int_{\mathbb{B}^{d-1}} \frac{f(y_1^{\perp}, \sqrt{ 1 - \| y_1^{\perp}\|^2}) + f(y_1^{\perp}, -\sqrt{ 1 - \| y_1^{\perp}\|^2})}{\sqrt{ 1 - \| y_1^{\perp}\|^2}} d y_1^{\perp} \\
& = \int_{\mathbb{B}^{d-1}} \frac{\sqrt{ 1 - \| y_1^{\perp}\|^2} + (-\sqrt{ 1 - \| y_1^{\perp}\|^2})}{\sqrt{ 1 - \| y_1^{\perp}\|^2}} d y_1^{\perp} = 0. \\
\end{align*}
It now follows from \eqref{eq:I_Q k = d+1} that $I_{Q_{d+1,1}^d}(\sigma) = 0$, so $\sigma$ minimizes $I_{Q_{d+1,1}^d}$.

When $k \leq d$, we need a bit more machinery. Let $Y_1,\ldots Y_{\dim(\mathcal{H}_l^{d-k+1})}$ be an orthonormal basis of $\mathcal{H}_l^{d-k+1}$, the space of spherical harmonics of degree $l$ on $\mathbb{S}^{d-k+1}$. Then the addition formula (see \cite[Theorem 1.2.6]{DX}) tells us that
\begin{equation}\label{eq:addition formula for Q}
Q_{k,l}^{d}( x_1, \ldots, x_k) = \det(W)^l \|  y_1  \|^{l} \|  y_2 \|^{l} \frac{1}{\dim(\mathcal{H}_l^{d-k+1})} \sum_{j=1}^{\dim(\mathcal{H}_l^{d-k+1})} Y_j( z_1) Y_j(z_2).
\end{equation}
Thus for any fixed $x_3, \ldots, x_k$ and $\mu \in \mathcal{M}(\mathbb{S}^{d-1})$, 
\begin{align*}
I_{U_{Q_{k,l}^d}^{x_3, \ldots, x_k}}(\mu) & = \int_{\mathbb{S}^{d-1}} \int_{\mathbb{S}^{d-1}} Q_{k,l}^d( x_1, x_2, \ldots, x_k)\, d\mu(x_1) d\mu(x_2) \\
& = \int_{\mathbb{S}^{d-1}} \int_{\mathbb{S}^{d-1}}\det(W)^l \|  y_1  \|^{l} \|  y_2 \|^{l} \sum_{j=1}^{\dim(\mathcal{H}_l^{d-k+1})} Y_j( z_1) Y_j(z_2) \,d\mu(x_1) d\mu(x_2) \\
& = \det(W)^l \sum_{j=1}^{\dim(\mathcal{H}_l^{d-k+1})} \Bigg( \int_{\mathbb{S}^{d-1}} Y_j( z_1) \|  y_1  \|^{l} d \mu(x_1) \Bigg)^2 \geq 0.
\end{align*}
Note that since $W$ is the Gram matrix of $x_3, \ldots, x_k$, its determinant is nonnegative. Thus, $Q_{k,l}^d$ is indeed $k$-positive definite, so for any $\mu \in \mathcal{P}(\mathbb{S}^{d-1})$, $I_{Q_{k,l}^d}(\mu) \geq 0$.

We now show that $\sigma$ minimizes the energy $I_{Q_{k,l}^d}$. For any fixed $x_3, \ldots, x_k \in \mathbb{S}^{d-1}$, we see that
$$I_{U_{Q_{k,l}^d}^{x_3, \ldots, x_k}}(\sigma) = \det(W)^l \sum_{j=1}^{\dim(\mathcal{H}_l^{d-k+1})} \Bigg( \int_{\mathbb{S}^{d-1}} Y_j( z_1) \|  y_1  \|^{l} d \sigma(x_1) \Bigg)^2.$$
If $x_3, \ldots, x_k$ are linearly dependent, we know this is zero. Assume that $x_3, \ldots, x_k$ are linearly independent, and for $1 \leq j \leq  \dim(\mathcal{H}_l^{d-k+1})$, let
$$f_j(x_1) = f_j( y_1^{\perp}, y_1) = Y_j( z_1) \|  y_1  \|^{l} .$$
By Lemma A.5.4 of \cite{DX}, we have that
\begin{align*}
\int_{\mathbb{S}^{d-1}} f_j(x_1) d \sigma (x_1) & = \int_{\mathbb{B}^{k-2}} ( 1 - \|y_1^{\perp} \|^2 )^{\frac{d-k}{2}} \left[ \int_{\mathbb{S}^{d-k+1}} f_j( y_1^{\perp} , \sqrt{1 - \| y_1^{\perp} \|^2} \xi ) d \sigma (\xi) \right] dy_1^{\perp} \\
& = \int_{\mathbb{B}^{k-2}} ( 1 - \| y_1^{\perp} \|^2 )^{\frac{d-k}{2}} \left[ \int_{\mathbb{S}^{d-k+1}} Y_j(\xi) (1 - \| y_1^{\perp} \|^2)^{\frac{l}{2}}  d \sigma (\xi) \right] dy_1^{\perp} \\
& = 0.
\end{align*}
Thus, for any fixed $x_3, \ldots, x_k \in \mathbb{S}^{d-1}$,
$$ \int_{ \mathbb{S}^{d-1}} \int_{\mathbb{S}^{d-1}} Q_{k,l}^d(x_1, x_2, \ldots, x_k) d \sigma(x_1) d \sigma(x_2) = 0,$$
meaning that
$$ I_{Q_{k,l}^d}(\sigma) = 0,$$
so $\sigma$ is indeed a minimizer of $I_{Q_{k,l}^d}$.

\end{proof}
We note here that $Q_{k,l}^d$ is zero if $x_3, \ldots, x_k$ are linearly dependent or if $x_1$ or $x_2$ are in $X$.

For $k=3$, these kernels are essentially \eqref{eq:BachocValQKernel}, introduced by Bachoc and Vallentin \cite{BV}. In this instance, note that $\det(W) =1$ is constant. The general case was covered by Musin \cite{Mu} who used the kernels to formulate general SDP bounds for spherical codes and, with some additional machinery, generalized the result of Schoenberg \cite{Sch} to characterize all positive definite kernels invariant under the stabilizer of $X$. However, in that setting, it was assumed that $x_3, \ldots, x_k$ were fixed and linearly independent, so no factor such as $\det(W)^l$ was included and the functions were only really functions of two variables. Recently, similar kernels with $k\geq 4$ were used for finding new bounds for sizes of equiangular sets of lines in \cite{DMOV}, where kernels were constructed in a way that assumed that distance set the last $k-2$ inputs had finitely many values, making them multivariate functions, but not allowing the last $k-2$ inputs to take arbitrary values on the sphere. The authors of \cite{DMOV} even discuss the difficulty of such a task. Our inclusion of $\det(W)$ as a factor of $Q$ allows us to address this issue, though alone, this would not allow us to construct functions which are not constant when $x_3, \ldots, x_k$ are linearly dependent, or more complicated positive definite functions, such as semidefinite combinations of the functions \eqref{eq:SemDefProgYs}. We discuss how to construct such functions later in this section.

For the main result of Section \ref{sec:A^2 on Sphere}, it is sufficient to use the case $l=1$ so we formulate the relevant statement as a separate lemma.

\begin{lemma}\label{lem:Musin-1}
For any set of fixed vectors $x_3, \ldots, x_k\in\mathbb{S}^{d-1}$, the kernel
$$Q_{k,1}^d (x_1, \ldots, x_k) = \det(W) \langle y_1, y_2 \rangle =  \det(W)u_{1,2} - w_1^T \adj(W) w_2  $$
is $k$-positive definite, and $I_{Q_{k,1}^d}$ is minimized by $\sigma$.
\end{lemma}

For small values of $k$, these kernels take the form:
\begin{align*}
Q_{3,1}^d &= u_{1,2} - u_{1,3}u_{2,3},\\
Q_{4,1}^d &= u_{1,2} -u_{1,2}u_{3,4}^2 - u_{1,3}u_{2,3} - u_{1,4}u_{2,4} + u_{1,3}u_{2,4}u_{3,4} + u_{1,4}u_{2,3}u_{3,4}.
\end{align*}

We can use these kernels $Q_{k,l}^d$ to construct various other kernels which are $k$-positive definite and whose energies are minimized by $\sigma$. Similar objects were studied in \cite{KV}.

\begin{corollary}\label{cor:MoreGenPosDef}
Let $G: (\mathbb{S}^{d-1} )^{k-1} \rightarrow \mathbb{R}$ be a continuous function such that, for $\eta_1, \eta_2, \ldots, \eta_{k-1} \in \mathbb{S}^{d-1}$, $G(\eta_1, \ldots, \eta_{k-1})$ depends only on the inner products $\langle \eta_i, \eta_j \rangle$, $1 \leq i < j \leq k-1$. Then the kernel
\begin{equation}\label{eq:MorGenPosDef}
T( x_1, x_2, \ldots, x_k) = G(x_1, x_3, \ldots, x_k) G( x_2, x_3,\ldots, x_k) Q_{k,l}^d( x_1, x_2, \ldots, x_k )
\end{equation}
is rotationally-invariant and $k$-positive definite. If $l \geq 1$, $T$ satisfies
\begin{equation}
\inf_{\mu \in \mathcal{P}(\mathbb{S}^{d-1})} I_{T}(\mu) = I_{T}(\sigma) = 0.
\end{equation}
\end{corollary}

From the way we defined $T$, we can see that $T$ is indeed continuous and symmetric in the first two variables.

\begin{proof}
We will use the same notation as in the proof of Theorem \ref{thm:Qdef}. We see immediately that the rotational-invariance of $T$ follows from the rotational-invariance of $Q_{k,l}^d$ and the inner products $\langle x_i, x_j \rangle$.  We also have that for fixed $x_3, \ldots, x_k$, $G(x_i, x_3, \ldots, x_k)$ depends only on $y_i^{\perp}$, the orthogonal projection of $x_i$ onto $X$.

For $k \leq d$, that $T$ is $k$-positive definite can be seen by the fact that for fixed $x_3, \ldots, x_k$ and $\mu \in \mathcal{M}(\mathbb{S}^{d-1})$, \eqref{eq:addition formula for Q} gives us
$$I_{U_T^{x_3, \ldots, x_k}}(\mu) = \det(W)^l \sum_{j=1}^{\dim(\mathcal{H}_l^{d-k+1})} \Bigg( \int_{\mathbb{S}^{d-1}} Y_j( z_1) \|  y_1  \|^{l} G( x_1, x_3,\ldots x_k) d \mu(x_1) \Bigg)^2 \geq 0.$$

If $ x_3, \ldots, x_k$ are linearly dependent, then $T = 0$, so assume that $x_3, \ldots, x_k$ are linearly independent,  and for $1 \leq j \leq  \dim(\mathcal{H}_l^{d-k+1})$, let
$$f_j(x_1) = f_j( y_1^{\perp}, y_1) = Y_j( z_1) \|  y_1  \|^{l} G( x_1, x_3,\ldots x_k).$$
By Lemma A.5.4 of \cite{DX}, and since $G$ does not depend on $y_1$, we have
\begin{align*}
\int_{\mathbb{S}^{d-1}} f_j(x_1) d \sigma (x_1) & = \int_{\mathbb{B}^{k-2}} ( 1 - \| y_1^{\perp} \|^2 )^{\frac{d-k}{2}} \left[ \int_{\mathbb{S}^{d-k+1}} f_j( y_1^{\perp} , \sqrt{1 - \| y_1^{\perp} \|^2} \xi ) d \sigma (\xi) \right] dy_1^{\perp} \\
& = \int_{\mathbb{B}^{k-2}} ( 1 - \| y_1^{\perp} \|^2 )^{\frac{d-k}{2}} (1 - \| y_1^{\perp} \|^2)^{\frac{l}{2}} G(x_1, x_3,\ldots x_k) \left[ \int_{\mathbb{S}^{d-k+1}} Y_j(\xi)   d \sigma (\xi) \right] d y_1^{\perp} \\
& = 0.
\end{align*}
Thus, for any fixed $x_3, \ldots, x_k \in \mathbb{S}^{d-1}$,
$$ \int_{ \mathbb{S}^{d-1}} \int_{\mathbb{S}^{d-1}} T(x_1, x_2, \ldots, x_k) d \sigma(x_1) d \sigma(x_2) = 0,$$
meaning that
$$ I_{T}(\sigma) = 0,$$
so $\sigma$ is indeed a minimizer of $I_{T}$.

The case of $k = d+1$ is similar.

\end{proof}

\begin{lemma}\label{lem:l=0PosDef}
Let $G: \big( \mathbb{S}^{d-1} \big)^{k-1} \rightarrow \mathbb{R}$ be continuous, depend only on the inner products of its inputs, and satisfy $\int_{\mathbb{S}^{d-1}} G( \eta_1, \ldots, \eta_{k-1}) d \sigma(\eta_1) = 0$. Then the kernel
\begin{equation}\label{eq:l=0PosDef}
H(x_1, x_2, \ldots, x_k) = G(x_1, x_3, \ldots, x_k) G( x_2, x_3, \ldots, x_k)
\end{equation}
is rotationally-invariant, $k$-positive definite, and satisfies
\begin{equation}
\inf_{\mu \in \mathcal{P}(\mathbb{S}^{d-1})} I_H(\mu) = I_H(\sigma) = 0.
\end{equation}
\end{lemma}

The formulation of $T$ and $H$ in the corollary and lemma, and the fact that the sum of $k$-positive definite kernels minimized by $\sigma$ is a $k$-positive definite kernel minimized by $\sigma$, allows us to now recover Theorem \ref{thm:SemiDefMin}. In \cite{BV}, the authors created matrices $Y_l^d$ of polynomials, and then took the trace of the product of a positive semidefinite matrix and a $Y_l^d$. When $l = 0$, this would lead to a sum of kernels of the form \eqref{eq:l=0PosDef}, and for $l > 0$ this would lead to a sum of kernels of the form \eqref{eq:MorGenPosDef}.

By combining  Lemmas \ref{lem:Schur's Lemma}, \ref{lem:Schur's Lemma2}, \ref{lem:kPosDef to nPosDef}, and \ref{lem:l=0PosDef} with Corollary \ref{cor:MoreGenPosDef} we can now construct a wide range of rotationally-invariant $k$-positive definite kernels whose energies are minimized by $\sigma$ from the kernels $Q_{n,l}^d$'s for $n < k$. In particular, we can construct kernels which are not constant when $x_3, \ldots, x_n$ are linearly dependent, unlike $Q_{k,l}^d$.

\section{Maximizing the  integral of $A^2$ on the sphere}\label{sec:A^2 on Sphere}

We now turn to the last main results of the paper. As an analogue of the result by Cahill and Casazza (Theorem \ref{thm:volume-gen}), we solve the optimization problem for $A^2$, the square of the $(k-1)$-dimensional volume of the simplex,  for an arbitrary number of inputs $3\le k \le d+1$. We have already proved some partial cases of the theorem below: Theorem \ref{thm: triangle area squared max} (for the case $k=3$ and $d\ge 2$, i.e. the area of the triangle)  and Theorem \ref{thm:A^2maxGen} (for full-dimensional simplices, i.e. $k=d+1\ge 3$).  We would like to point out, that the latter theorem applies to measures on $\mathbb R^d$. The following  theorem, while  restricted to the sphere,  covers the whole range $3\le k \le d+1$.

\begin{theorem}\label{thm:area-gen} Let $d\ge 2$ and $3\leq k\leq d+1$.  Let $A(x_1, \ldots, x_k)$ be the $(k-1)$-dimensional Euclidean volume of a simplex with vertices $x_1, \ldots, x_k \in \mathbb{S}^{d-1}$.  Then the  set of maximizing measures of $I_{A^2}$ in $\mathcal{P}(\mathbb S^{d-1})$ is the set of balanced isotropic measures on $\mathbb S^{d-1}$. In particular, the uniform surface measure $\sigma$ maximizes $I_{A^2}$. The value of the maximum is $\frac {k}{(k-1)! d^{k-1}}\binom{d}{k-1}$.
\end{theorem}

\begin{proof}
Let $U$ be the Gram matrix of vectors $\{x_1,\ldots,x_k\}\subset \mathbb S^{d-1}$ with entries $u_{i,j}$, i.e. $\langle x_i,x_j \rangle = u_{i,j}$ for $1\leq i,j\leq k$. For $I,J\subseteq \{1,\ldots,k\}$, we denote by $U_{I,J}$ the submatrix of $U$ obtained by deleting rows with numbers from $I$ and columns with numbers from $J$.  By Lemma \ref{lem:A-formula}, whose proof is postponed to the Appendix, $$((k-1)!)^2 A^2 = - \det \begin{pmatrix}U&\mathbf{1}\\ \mathbf{1}^T&0\end{pmatrix}.$$ We expand the determinant by choosing, for each $i,j \in \{ 1,\ldots, k\}$, the elements in the last column and $i^{th}$ row and the $j^{th}$ column and last row. We treat the cases $i=j$ and $i\neq j$ separately.

\begin{align*}
((k-1)!)^2 A^2 & = - \sum\limits_{i=1}^k (-1)^{k+1+i+k+i} \det(U_{\{i\},\{i\}}) - \sum\limits_{i\neq j} (-1)^{k + 1 + i+ k +j} \det(U_{\{i\},\{j\}})\\
&=\sum\limits_{i=1}^k \det(U_{\{i\},\{i\}}) + \sum\limits_{i\neq j} (-1)^{i+j} \det(U_{\{i\},\{j\}})
\end{align*}

For each $i \in \{ 1, \ldots, k \}$, $\det(U_{\{i\},\{i\}})$ is the $(k-1)$-point kernel $V^2 (x_1,\dots, x_{i-1},x_{i+1},\dots,x_n)$ from Theorem \ref{thm:volume-gen}. Subsequently, Theorem \ref{thm:volume-gen} implies that the energy integral  for the kernel defined by the first sum is not greater than $k \frac {(k-1)!} {d^{k-1}} \binom{d}{k-1}$.

It is now sufficient to show that the contribution of the second sum is nonpositive. Let us fix $i,j \in \{1, \ldots, k\}$, with $i \neq j$, and denote $U_{\{i,j\},\{i,j\}}$ by $U'$. We expand $\det(U_{\{i\},\{j\}})$ by row $j$ of $U$ and column $i$ of $U$ taking an element $u_{j,m}$, $m\neq j$, from the row and $u_{n,i}$, $n\neq i$, from the column, respectively. 

If $m = i$ and $n = j$, then we take $u_{j,i}$ both for the row and the column expansion. The contribution of this case to $\det(U_{\{i \}, \{ j\}})$ the sum is then $(-1)^{i+j -1} u_{j,i} \det(U')$. 

Let us now consider the case where $m\neq i$ and $n\neq j$. Without loss of generality, let us assume that $i < j$ (the case of $i > j$ is similar). Let $n'$ be the position of row $n$ of $U$ after rows $i$ and $j$ are deleted, i.e. $n'=n$ if $n<i$, $n'=n-1$ if $i < n < j$, and $n' = n-2$ if $j < n$. Similarly we define $m'=m$ if $m<i$, $m'=m-1$ if $i < m < j$, and $m' = m-2$ if $j < m$. This guarantees that  $U_{\{i,j,n\},\{i,j,m\} } = U'_{ \{n'\},\{m'\}}$. A careful examination of the signs  shows that the contribution of this expansion in the sum is then
\begin{align*}
(-1)^{i +n' + j + m'} u_{n,i} u_{j, m} \det(U'_{\{n'\},\{m'\}}) &= (-1)^{p} u_{n,i} u_{j, m} \det(U'_{\{n'\},\{m'\}}) \\ &= (-1)^{ p} u_{n,i} u_{j, m} \det(U_{\{i, j, n\},\{i, j, m\}})
\end{align*}
where
$$p = i + n + j + m + \frac{\sgn(n-i) + \sgn(n-j) + \sgn(m-i) + \sgn(m-j)}{2} - 2.$$

Overall, we have
\begin{align*}
(-1)^{i+j}\det(U_{\{i\},\{j\}}) & = (-1)^{2i + 2j -1}u_{j,i} \det(U') + \sum_{\substack{1\leq m,n\leq k \\ m,n\notin \{i,j\}}} (-1)^{2i + 2j + m' + n'} u_{n,i} u_{j,m}\det(U'_{\{n'\},\{m'\}}) \\
& = -u_{j,i} \det(U') + \sum_{\substack{1\leq m,n\leq k\\ m,n\neq i,j}} (-1)^{m'+n'}u_{n,i} u_{j,m} \det(U'_{\{n'\},\{m'\}}) \\
& = -(u_{j,i} \det(U') - {u_i'}^T \operatorname{adj}(U') u_j')\\
& = -(u_{j,i} \det(U') - {u_i'}^T \adj(U') u_j')\\
& = - Q_{k,1}^d( x_i, x_j, x_{l_1}, \ldots, x_{l_{k-2}}),
\end{align*}
where $u_i'=(u_{1,i},\ldots,u_{k,i})^T$ with the first index running through all $n\neq i,j$, $u_j'=(u_{j,1},\ldots,u_{j,k})^T$ with the second index running through all $m\neq i,j$, and $\{l_1, \ldots, l_{k-2} \} = \{ 1, \ldots, k \} \setminus \{i, j \}$. For the last identity above, see Lemma \ref{lem:Musin-1}. From Theorem \ref{thm:Qdef} (or Lemma \ref{lem:Musin-1} specifically for this case), we know that $-Q_{k,1}^d$ is $k$-negative definite, so its contribution to $I_{A^2}$, and therefore the contribution of
$$K (x_1, \ldots, x_k) = \sum_{i \neq j} (-1)^{i+j} \det(U_{\{i \}, \{j\}}) = - \frac{\binom{k}{2}}{k!} \sum_{\pi} Q_{k,1}^d( x_{\pi(1)}, \ldots, x_{\pi(k)}),$$
to $I_{A^2}$ is nonpositive, so $I_{A^2} \leq \frac{k}{(k-1)! d^{k-1}} \binom{d}{k-1}$.

It remains to find out which measures maximize $I_{A^2}$. Due to Theorem \ref{thm:volume-gen}, any maximizing measure $\mu$ must be isotropic. In order to find measures vanishing on the second part we need to return to Lemma \ref{lem:Musin-1}. The necessary and sufficient condition for vanishing on $Q_{k,1}^d(x_1, \ldots, x_k)$ is the following. For any linearly independent $x_3, \ldots, x_k$ from the support of $\mu$ and the linear space $X$ generated by them, the projection of $\mu$ onto $X^{\perp}$ must be balanced. In other words, the center of mass of $\mu$ must belong to $X$. An isotropic measure must be full-dimensional so there must exist $d$ linearly independent vectors from $\operatorname{supp}(\mu)$. The intersection of all linear spaces generated by any $(k-2)$ of these $d$ vectors is only the origin so the center of mass of $\mu$ must be at the origin. Clearly, balanced isotropic measures attain the found maximum.
\end{proof}

\begin{remark}
In the last part of the proof we could also have shown that $\sigma$ is a maximizer, then noted that the potential is a polynomial of degree at most two in any of its variables, with some parts being of degree one, meaning that balanced isotropic measures are the maximizers, since they yield the same value of energy as $\sigma$ (this a direct analogy to spherical $2$-designs). \end{remark}

We note that in the case that $k = 2$, $A(x,y)$ is simply the Euclidean distance between $x$ and $y$. If we were to split $\| x-y\|^2$ into a linear part and a ``volume" part as in the proof, then the volume is simply the distance from the origin to a point on the circle, which is always $1$. Thus, in that particular case, only the linear part matters, so maximizers of $I_{A^2}$ on the sphere are all balanced measures, as shown in \cite{Bj}. The case of $k = 3$ was handled by Theorem \ref{thm: triangle area squared max}, where in this case $Q_{k,1}^d = Y_{1,0,0}$ and the $(k-1)$-volume squared function is $V^2(x,y) = 1- \langle x, y \rangle^2$, as discussed in Section \ref{sec:k=2}.

Finally, we note that despite having this result for $A^2$ on the sphere, Corollary \ref{cor:A^s for s>2} does not hold if $A$ has $k < d+1$ inputs. As $s \rightarrow \infty$, we should expect the maximizer of $I_{A^s}$ to be supported on some set such that $A$ takes only its minimum (zero) and maximum values. This will be the set of vertices of a regular $(k-1)$-simplex on some $(k-2)$-dimensional subsphere. Such a measure is not isotropic on $\mathbb{S}^{d-1}$ for $k < d+1$, and so we can also not use the same proof method  to determine maximizers. We  do, however, conjecture that the maximizers of $I_{A^s}$ are discrete when $s>2$,  for all $3\le k \le d+1$. 

\section{Acknowledgments}

We would like to thank Danylo Radchenko  and David de Laat for fruitful discussions and useful suggestions.
All of the authors express gratitude to ICERM for hospitality and support during the Collaborate{@}ICERM  program  in  2021. D.~Bilyk has been supported by the NSF grant DMS-2054606 and Simons  Collaboration Grant 712810. D.~Ferizovi\'{c} thankfully acknowledges support by the Methusalem grant of the Flemish Government. A.~Glazyrin was supported by the NSF grant DMS-2054536. R.W.~Matzke was supported by the Doctoral Dissertation Fellowship of the University of Minnesota, the Austrian Science Fund FWF project F5503 part of the Special Research Program (SFB) ``Quasi-Monte Carlo Methods: Theory and Applications", and NSF Postdoctoral Fellowship Grant 2202877. O.~Vlasiuk was supported by an AMS-Simons Travel Grant.

\section{Appendix: Expressing \texorpdfstring{$A^2$}{A2} through Gram determinants}\label{sec:Appen A^2}

Let $U$ be the Gram matrix of vectors $\{x_1,\ldots,x_k\}\subset \mathbb S^{d-1}$ with entries $u_{i,j}$, i.e. $\langle x_i,x_j \rangle = u_{i,j}$ for $1\leq i,j\leq k$. The following lemma provides a linear-algebraic description of $A^2$.

\begin{lemma}\label{lem:A-formula}
$$A^2 (x_1,\ldots,x_k) = -\frac 1 {((k-1)!)^2} \det \begin{pmatrix}U&\mathbf{1}\\ \mathbf{1}^T&0\end{pmatrix},$$
where $\mathbf{1}$ is the column vector of $k$ ones.
\end{lemma}

\begin{proof}

$A^2$ can be found from the Gram matrix of the vectors $x_2-x_1, \ldots, x_k-x_1$.
\begin{align*}
A^2 (x_1,\ldots,x_k) &= \frac 1 {((k-1)!)^2} \det\begin{pmatrix}\langle x_2-x_1, x_2-x_1\rangle&\ldots&\langle x_2-x_1, x_k-x_1\rangle \\ \vdots & \ddots& \vdots \\ \langle x_k-x_1, x_2-x_1\rangle & \ldots & \langle x_k-x_1, x_k-x_1\rangle\end{pmatrix} \\
& = \frac 1 {((k-1)!)^2} \det\begin{pmatrix}2-2u_{1,2}&\ldots&1+u_{2,k} - u_{1,2} - u_{1,k} \\ \vdots & \ddots& \vdots \\ 1+u_{k,2} - u_{1,k} - u_{1,2}& \ldots & 2-2u_{1,k}\end{pmatrix} \\
& = \frac {-1} {((k-1)!)^2} \det\begin{pmatrix}0 & 0 & \ldots & 0 &1\\0 & 2-2u_{1,2}&\ldots&1+u_{2,k} - u_{1,2} - u_{1,k} & 0\\ \vdots & \vdots & \ddots& \vdots &\vdots\\0& 1+u_{k,2} - u_{1,k} - u_{1,2}& \ldots & 2-2u_{1,k}&0\\ 1 & 0 & \ldots & 0 & 0\end{pmatrix}\\
& = \frac {-1} {((k-1)!)^2} \det\begin{pmatrix}0 & u_{1,2} & \ldots & u_{1,k} &1\\u_{1,2} & 2-2u_{1,2}&\ldots&1+u_{2,k} - u_{1,2} - u_{1,k} & 0\\ \vdots & \vdots & \ddots& \vdots &\vdots\\u_{1,k}& 1+u_{k,2} - u_{1,k} - u_{1,2}& \ldots & 2-2u_{1,k}&0\\ 1 & 0 & \ldots & 0 & 0\end{pmatrix}.
\end{align*}

Note that in the third equality, we created a $(k+1) \times (k+1)$ matrix whose determinant is the negative of our original matrix, due to the only nonzero entries in the last row and column being the ones in the upper right and lower left corners. This also means that inserting $u_{1,j}$'s into the first row and column doesn't affect the determinant.

Now we add the  first row and column to all rows and columns except for the last ones.

$$A^2=-\frac 1 {((k-1)!)^2} \det\begin{pmatrix}0 & u_{1,2} & \ldots & u_{1,k} &1\\u_{1,2} & 2&\ldots&1+u_{2,k} & 1\\ \vdots & \vdots & \ddots& \vdots &\vdots\\u_{1,k}& 1+u_{k,2}& \ldots & 2&1\\ 1 & 1 & \ldots & 1 & 0\end{pmatrix}.$$

We subtract the last column from all columns except for the first one, then add the bottom row to the top row, and see that
\begin{align*}
A^2 & = -\frac 1 {((k-1)!)^2} \det\begin{pmatrix}0 & u_{1,2}-1 & \ldots & u_{1,k}-1 &1\\u_{1,2} & 1&\ldots&u_{2,k} & 1\\ \vdots & \vdots & \ddots& \vdots &\vdots\\u_{1,k}& u_{k,2}& \ldots & 1&1\\ 1 & 1 & \ldots & 1 & 0\end{pmatrix} \\
& = -\frac 1 {((k-1)!)^2} \det\begin{pmatrix}1 & u_{1,2}& \ldots & u_{1,k}&1\\u_{1,2} & 1&\ldots&u_{2,k} & 1\\ \vdots & \vdots & \ddots& \vdots &\vdots\\u_{1,k}& u_{k,2}& \ldots & 1&1\\ 1 & 1 & \ldots & 1 & 0\end{pmatrix}= -\frac 1 {((k-1)!)^2} \begin{pmatrix}U&\mathbf{1}\\ \mathbf{1}^T&0\end{pmatrix}.
\end{align*}

\end{proof}

\bibliographystyle{alpha}

\end{document}